\documentclass[11pt, letterpaper]{article}
\usepackage{geometry}                % See geometry.pdf to learn the layout options. There are lots.
\usepackage{graphicx}
\usepackage{amssymb}
\usepackage{epstopdf}
\usepackage[T1]{fontenc}
\usepackage[utf8]{inputenc}
\usepackage[ngerman, english]{babel}
\usepackage{mathtools}
\usepackage{bbm}
\usepackage{mathrsfs}
\usepackage{amsmath,amscd}
\usepackage{tikz}
\usepackage{tikz-cd}
\usepackage{babel}
\usepackage{mathalfa}
\usetikzlibrary{arrows, matrix}
\usetikzlibrary{cd}
\usepackage[arrow, matrix, curve]{xy}

\newcommand{\qed}{\hfill \ensuremath{\Box}}
\newenvironment{proof}{\vspace{1ex}\noindent{\it Proof.}\hspace{0.5em}}
	{\hfill\qed\vspace{1ex}}

\newtheorem{theorem}{Theorem}[section]
\newtheorem{lemma}[theorem]{Lemma}
\newtheorem{proposition}[theorem]{Proposition}
\newtheorem{corollary}[theorem]{Corollary}

\newtheorem{definition}[theorem]{Definition}

\DeclareMathOperator{\Gal}{\operatorname{Gal}}
\DeclareMathOperator{\Q}{\mathbf{Q}}

\DeclareMathOperator{\Z}{\mathbf{Z}}

\DeclareMathOperator{\N}{\mathbf{N}}

\DeclareMathOperator{\Spec}{\operatorname{Spec}}

\DeclareMathOperator{\Hom}{\operatorname{Hom}}
\DeclareMathOperator{\Frac}{\operatorname{Frac}}

\DeclareMathOperator{\Lie}{\mathrm{Lie}}

\DeclareMathOperator{\Ext}{\operatorname{Ext}}

\DeclareMathOperator{\Og}{\mathcal{O}}

\DeclareMathOperator{\Pic}{\mathrm{Pic}}

\DeclareMathOperator{\Br}{\mathrm{Br}}

\DeclareMathOperator{\Gm}{\mathbf{G}_m}
\DeclareMathOperator{\NGm}{\mathscr{G}_m}
\DeclareMathOperator{\Ga}{\mathbf{G}_a}

\DeclareMathOperator{\Res}{\mathrm{Res}}

\DeclareMathOperator{\alg}{{^\mathrm{alg}}}
\DeclareMathOperator{\sep}{{^\mathrm{sep}}}

\title{Chai's conjecture for semiabelian Jacobians}
\author{Otto Overkamp}
\date{}
\begin{document}
\maketitle
%\blfootnote{2020 Mathematics Subject Classification [MSC]: 11G20, 14L15. Keywords: Néron models, Picard schemes, Base change conductor.}
{\abstract{We prove Chai's conjecture on the additivity of the base change conductor of semiabelian varieties in the case of Jacobians of proper curves. This includes the first infinite family of non-trivial wildly ramified examples. Along the way, we extend Raynaud's construction of the Néron lft-model of a Jacobian in terms of the Picard functor to arbitrary seminormal curves (beyond which Jacobians admit no Néron lft-models, as shown by our more general structural results).
Finally, we investigate the structure of Jacobians of (not necessarily geometrically reduced) proper curves over fields of degree of imperfection at most one and prove two conjectures about the existence of Néron models and Néron lft-models due to Bosch-Lütkebohmert-Raynaud for Jacobians of general proper curves in the case of perfect residue fields, thus strengthening the author's previous results in this situation.}}
\tableofcontents
\section{Introduction}
Let $\Og_K$ be an Henselian discrete valuation ring with perfect residue field $k$ and field of fractions $K.$ Let $G$ be a semiabelian variety over $K;$ i. e. suppose that $G$ is a smooth connected algebraic group over $K$ which fits into an exact sequence
$$0 \to T\to G\to E \to 0,$$
where $E$ is an Abelian variety and $T$ an algebraic torus over $K,$ respectively. A \it Néron lft-model \rm of $G$ is a smooth separated model $\mathscr{N}$ of $G$ over $\Og_K$ such that for any smooth $\Og_K$-scheme $S,$ any $K$-morphism $S_K\to G$ uniquely extends to an $\Og_K$-morphism $S\to \mathscr{N}.$ These group schemes are among the most well-studied objects in arithmetic geometry due to their fundamental importance (see \cite{BLR} for a general introduction). However, questions regarding their behaviour are usually very delicate. The two problems most often encountered are that Néron lft-models do not in general commute with ramified base change, and that the sequence 
$$0 \to \mathscr{T} \to \mathscr{N} \to \widetilde{\mathscr{N}} \to 0$$ of Néron lft-models induced by the sequence above is usually not exact. \\
Recall that $G$ is said to have \it semiabelian reduction \rm if the special fibre of $\mathscr{N}^0$ is an extension of an Abelian variety by a torus over $k.$ It is well-known that there exists a finite separable extension $L$ of $K$ such that $G_L$ has semiabelian reduction over the integral closure $\Og_L$ of $\Og_K$ in $L;$ we denote its Néron lft-model by $\mathscr{N}_L.$ Following Chai \cite{Chai}, we consider the \it base change conductor \rm
$$c(G):=\frac{1}{e_{L/K}}\ell_{\Og_L}(\mathrm{coker}((\Lie\mathscr{N})\otimes_{\Og_K}\Og_L \to \Lie\mathscr{N}_L)),$$ 
where $e_{L/K}$ is the ramification index of the extension $K\subseteq L.$ The base change conductor is a rational number which measures the failure of $\mathscr{N}^0$ to commute with base change along $\Og_K\subseteq \Og_L$ (or, equivalently, the failure of $G$ to have semiabelian reduction). Moreover, it does not depend upon the choice of $L.$ \\
In \cite[(8.1)]{Chai}, Chai conjectured that, given an exact sequence as above, we have
$$c(G)=c(T)+c(E).$$
The conjecture has generated considerable interest and has led, for example, to new discoveries in the field of motivic integration \cite{CLN}.  
It has been proven by Chai under the assumption that $k$ be finite or that $K$ be of characteristic zero \cite[Theorem 4.1]{Chai}; the latter case was later re-proven in \cite[Theorem 4.2.4]{CLN}.  Furthermore, the conjecture can be shown by relatively elementary means if $L$ may be chosen to be tamely ramified over $K$ (see the introduction to \cite{CLN} for more details). \\
The purpose of the present paper is to prove Chai's conjecture for \it Jacobians. \rm More precisely, we shall prove
\begin{theorem}\rm (Theorem \ref{Chaithm}) \it 
Let $C$ be a proper curve\footnote{i. e. purely one-dimensional scheme} over $K$ and assume that $G:=\Pic^0_{C/K}$ is semiabelian. Then Chai's conjecture holds for $G.$ 
\end{theorem}
Throughout the paper, we shall assume that $\Og_K$ is complete, that $k$ is algebraically closed, and that $\mathrm{char}\, K = p>0.$ This leads to no loss of generality by Chai's result in characteristic 0 (recalled above), together with \cite[Chapter 10.1, Proposition 3]{BLR}. Moreover, except for the last paragraph, we shall assume that $C$ is reduced, which is no restriction either because, assuming that $\Pic^0_{C/K}$ is semiabelian, the morphism $\Pic^0_{C/K} \to \Pic^0_{C_{\mathrm{red}}/K}$ is an isomorphism \cite[Chapter 9.2, Proposition 5]{BLR}. \\
The tools we shall employ come from two sources: On the one hand, we shall generalise a result originally due to Raynaud \cite[Theorem 3.1]{LLR} and adapt the methods introduced in \it op. cit. \rm to our situation. In particular, we obtain a description of the Néron lft-model of the Jacobian of a proper seminormal curve in terms of the Picard functor of a suitably chosen model of the curve (see Theorem \ref{Raynaudtypethm} and Corollary \ref{RaynaudGenCor}), generalising Raynaud's results \cite{RaynaudPic} as far as possible. On the other hand, we shall make use of the methods developed by the author to study Néron models of Jacobians of singular curves \cite{OvI,Ov}. \\
While our results do not solve the conjecture in full generality, the condition we impose is completely different from those required by any of the other partial results currently known. In particular, there are no conditions on either $K$ or $k$ (other than the latter's perfectness, which is already present in Chai's original conjecture), or the ramification of the extension $L/K.$ The conjecture is trivial whenever the sequence $0\to \mathscr{T} \to \mathscr{N} \to \widetilde{\mathscr{N}}\to 0$ is exact; we include an example which shows that this is not the case in general for Jacobians (see Paragraph \ref{examplepar}). In particular, it seems that our proof cannot be simplified significantly using the present methods. \\
Finally, we shall use this opportunity to show how the main results from \cite{Ov} (particularly the two conjectures due to Bosch-Lütkebohmert-Raynaud) can be proven for Jacobians without the condition of geometric reducedness imposed throughout \it op. cit. \rm if the residue fields of the (Dedekind) base scheme $S$ are perfect. \\
\\
$\mathbf{Acknowledgement}.$ The author would like to thank the Mathematical Institute of the University of Oxford, where this paper was written, for its hospitality. He would like to express his gratitude to Professor D. Rössler for helpful conversations. Moreover, the author was supported by the German Research Foundation (Deutsche Forschungsgemeinschaft; Geschäftszeichen OV 163/1-1, Projektnummer 442615504), for whose contribution he is most grateful. Finally, the author would like to express his gratitude to the referee for making a number of very valuable comments on a previous version of this article, and in particular for pointing out that Lemma \ref{alternatingsplitlem} was already known.

\section{Preliminary results}

\subsection{Excellence and degrees of imperfection}
Let $R$ be a discrete valuation ring of equal characteristic $p>0.$ Let $K$ and $\kappa$ be the field of fractions and the residue field of $R,$ respectively. For any field $F$ of characteristic $p,$ we let $\delta(F)$ be the unique nonnegative integer (or symbol $\infty$) such that $[F:F^p]=p^{\delta(F)}$ and call it the \it degree of imperfection \rm of $F.$ Clearly, $F$ is perfect if and only if $\delta(F)=0.$ If $R$ is excellent, the numbers $\delta(K)$ and $\delta(\kappa)$ are related as follows:
\begin{lemma}
Suppose that $R$ is excellent. Then $\delta(K) = \delta(\kappa)+1.$\label{plusonelem}
\end{lemma}
\begin{proof}
We begin by treating the case where $\delta(K)<\infty$ and claim that, in this case, the Frobenius morphism $R\to R,$ $r\mapsto r^p$ is finite. This is the same as showing that $R$ is finite over $R^p.$ As $R$ is integrally closed and integral over $R^p,$ $R$ is the integral closure of $R^p$ in $K.$ As $R$ is excellent and $K$ is finite over $K^p,$ this shows that $R$ is finite over $R^p.$ The fundamental equation \cite[Tag 02MJ]{Stacks} shows that $[\kappa:\kappa^p]$ is finite and
$$[K:K^p]=p[\kappa:\kappa^p],$$ as the Frobenius map $R\to R$ clearly has ramification index $p.$ This implies the claim from the Lemma in this case. \\
If $\delta(K)=\infty,$ all we must show is that $\delta(\kappa)=\infty$ as well. Let $K^p\subseteq L \subseteq K$ be a finite subextension, and let $R_L$ be the integral closure of $R^p$ in $L.$ If $\pi$ denotes a uniformiser of $R^p$ and $\pi_L$ a uniformiser\footnote{Because $R^p$ is excellent and $L$ is finite and purely inseparable over $K^p,$ $R_L$ is a discrete valuation ring.}  of $R_L,$ then $\pi_L^p=\epsilon \pi^\nu$ for some $\epsilon\in R^{p,\times}$ and $\nu\in \N;$ moreover we have $\pi=\eta\pi_L^{e_L}$ for some $\eta\in R_L^\times,$ $e_L\in \N.$ Together these equations imply $\pi_L^p=\epsilon \eta^\nu \pi_L^{e_L\nu},$ so $e_L\in \{1,p\}.$ For each $K^p\subseteq L\subseteq K,$ let $\mathfrak{m}_L$ be the maximal ideal of $R_L.$ As above, we have the equation $[L:K^p]=e_L[R_L/\mathfrak{m}_L:\kappa^p].$ Then
$$\kappa=\varinjlim_L R_L/\mathfrak{m}_L$$ as an extension of $\kappa^p.$ However, as $[K:K^p]=\infty,$ the integer $[L:K]$ becomes arbitrarily large as $L$ runs through finite subextensions of $K^p \subseteq K.$ Since $e_L$ is bounded independently of $L,$ the number $[R_L/\mathfrak{m}_L:\kappa^p]$ must become arbitrarily large as well, which means that $[\kappa:\kappa^p]=\infty.$ 
\end{proof}
\begin{corollary}
Let $S$ be an excellent Dedekind scheme of equal characteristic $p>0$ with field of rational functions $K.$ Then the following are equivalent: \\
(i) there exists a closed point $x \in S$ such that $\kappa(x)$ is perfect,\\
(ii) for all closed points $x\in S,$ $\kappa(x)$ is perfect, and \\
(iii) we have $\delta(K)=1.$ \label{Deltacor}
\end{corollary}
\begin{proof}
$(i) \Rightarrow (iii)$ follows from Lemma \ref{plusonelem}, as does $(iii)  \Rightarrow (ii)$ (both by considering $\Og_{S,x}$), and $(ii) \Rightarrow (i)$ is trivial. 
\end{proof} \\
We would like to remark that, using Lemma \ref{plusonelem}, we can give a procedure for constructing non-excellent discrete valuation rings which is slightly easier than that due to Orgogozo (see Raynaud's exposition \cite{Raynaud}), although the construction we use is the same:
\begin{proposition} \rm(cf. \cite[Proposition 11.6]{Raynaud}) \it
Let $k$ be a perfect field of characteristic $p>0$ and let $k\subseteq L \subseteq k(\!(t)\!)$ be a finitely generated subextension of transcendence degree $\geq 2.$ Then the ring
$$R:=L\cap k[\![t]\!]$$ is a non-excellent discrete valuation ring. 
\end{proposition}
\begin{proof}
We can write $L$ as a finite separable extension of a purely transcendental extension $T$ of $k$ of transcendence degree $n\geq 2.$ Since $L=L^p\otimes_{T^p}T$ (by the primitive element theorem), we see that $\delta(T)=\delta(L).$ Moreover, clearly $\delta(T)=n$ since $k$ is perfect. If $\mathfrak{m}$ denotes the maximal ideal of $R,$ we have $\mathfrak{m}=L\cap \langle t \rangle,$ so we have $k\subseteq R/\mathfrak{m} \subseteq k[\![t]\!]/\langle t\rangle=k.$ This shows that $\delta(R/\mathfrak{m})=0$ whereas $\delta(L)\geq 2$ by assumption. Hence $R$ is not excellent by Lemma \ref{plusonelem}.   
\end{proof}
\begin{lemma}
Let $\kappa \subseteq \ell$ be an algebraic extension of fields of characteristic $p>0$ which we assume to be separable or finite. Then $\delta(\ell)=\delta(\kappa).$ \label{deltafinsepeqlem}
\end{lemma}
\begin{proof}
Assume first that $\kappa\subseteq \ell$ is finite. Then we have the towers of extensions $\kappa^p\subseteq \ell^p \subseteq \ell$ and $\kappa^p\subseteq \kappa\subseteq \ell.$ The tower law shows that $[\ell:\ell^p][\ell^p:\kappa^p]=[\kappa:\kappa^p][\ell:\kappa].$ Since clearly $[\ell:\kappa]=[\ell^p:\kappa^p],$ the desired equality follows in this case. If $\ell$ is separable over $\kappa,$ then for any finite subextension $\kappa\subseteq \ell'\subseteq \ell,$ we have $\ell'=\ell'^p\otimes_{\kappa^p} \kappa.$ This follows from the primitive element theorem. But since $\ell$ is the inductive limit of its finite subextensions and since tensor products commute with inductive limits, we have $\ell=\ell^p\otimes_{\kappa^p} \kappa,$ which clearly implies the claim. 
\end{proof}
\subsection{Stein reduced morphisms}
Let $f\colon X\to S$ be a proper morphism of schemes with $S$ locally Noetherian. Then $f$ factors into a morphism $f'\colon X\to \Spec f_\ast \Og_X$ and the canonical projection $\Spec f_\ast \Og_X\to S.$ Moreover, the morphism $f'$ has geometrically connected fibres (this is Stein factorisation \cite[Tag 03H0]{Stacks}). We begin by making the following
\begin{definition}
Let $f\colon X\to S$ be a proper morphism of schemes with $S$ locally Noetherian. We say that $f$ is \rm Stein reduced \it if the induced map $f'$ has geometrically reduced fibres.
\end{definition}
Even if $S=\Spec K$ for a field $K,$ a general proper morphism $f\colon X\to S$ need not be Stein reduced (this is the well-known phenomenon that $X$ need not have geometrically reduced fibres over $\Gamma(X, \Og_X)$). One of the main observations of this paragraph is that if $f\colon X\to \Spec K$ is a proper morphism such that $X$ is normal and $\delta(K)\leq 1,$ then $f$ is Stein reduced. The argument we shall give is essentially due to Schröer (see \cite[proof of Lemma 1.3]{Schr}). 
\begin{proposition}
Let $f\colon X\to S$ be a proper morphism with $S$ locally Noetherian. Suppose that $f$ is cohomologically flat in dimension zero and has normal fibres. Moreover, assume that all points $s\in S$ satisfy $\delta(\kappa(s))\leq 1.$ Then $f$ is Stein reduced. \label{Steinredprop}
\end{proposition}
\begin{proof}
We may assume without loss of generality that $S=\Spec K$ for some field with $\delta(K)\leq 1.$ If $K$ is perfect, all proper morphisms $f\colon X\to \Spec K$ with $X$ reduced are Stein reduced, so we shall assume $\delta(K)=1.$ Moreover, we may assume that $X$ is connected, so that $\Gamma(X, \Og_X)$ is a field, and replace $K$ by $\Gamma(X, \Og_X).$ All we must show is that $X\times_K \Spec K^{1/p}$ is reduced \cite[Proposition 1.2]{Schr}. Since $K$ has degree of imperfection 1 (and hence so does $\Gamma(X, \Og_X)$ by Lemma \ref{deltafinsepeqlem}), we may pick any $\gamma\in K^{1/p}\backslash K$ and obtain an isomorphism $K^{1/p}\cong K[X]/\langle X^p-\gamma^p \rangle.$ Now let $F$ be the field of rational functions on $X.$ If $F\otimes_KK^{1/p}$ were non-reduced, there would have to exist an element $\xi\in F$ such that $\xi^p=\gamma^p.$ Now cover $X$ by open affine subsets $U_1,..., U_n.$ Since $\xi$ is clearly integral over $\Gamma(U_j,\Og_{U_j})$ for all $j$ and since $X$ is normal, we see that $\xi\in \Gamma(U_j, \Og_{U_j})$ for all $j.$ Hence $\xi\in K\subseteq K^{1/p},$ so $\xi=\gamma,$ which contradicts our choice of $\gamma.$ 
\end{proof}
\subsection{General background}
We recall a few basic results which we shall use freely throughout the paper. For a general introduction to the theory of Picard functors and rigidificators, see \cite{BLR} or \cite{RaynaudPic}. For a general introduction to Néron (lft-)models, see \cite{BLR}. We shall make extensive use of the results from \cite{LLR} and our strategy is modelled in part on this paper. Some results in \cite{LLR} are not correct as stated (see \cite{LLRII}), but this does not affect any of the results from \cite{LLR} which we shall use. All group schemes in this article will be commutative.
\begin{itemize}
\item If $S$ is a Dedekind scheme, $\mathscr{G}$ is a group scheme locally of finite type over $S$ and $\mathscr{H}$ is a closed subgroup scheme of $\mathscr{G}$ which is flat over $S,$ then the fppf-quotient $\mathscr{G}/\mathscr{H}$ is representable by a group scheme which is locally of finite type over $S$ \cite[Théorème 4.C]{An}.
\item If $S'$ is a finite flat extension of $S$ and $\mathscr{G'} \to S'$ is a separated group scheme locally of finite type over $S',$ then $\Res_{S'/S}\mathscr{G'}$ is representable by a group scheme locally of finite type over $S,$ which is smooth (resp. étale) over $S$ if $\mathscr{G'}\to S'$ is smooth (resp. étale); see \cite[ Lemma 2.7]{Ov}. If $\mathscr{G'}\to S'$ is smooth and has connected fibres then so does $\Res_{S'/S}\mathscr{G}$ (combine \cite[Chapter 6.4, Theorem 1]{BLR} with \cite[Proposition A.5.9]{CGP}).
\item If $\Spec \mathcal{A} \to \Spec \mathcal{B} \to S$ are finite and faithfully flat morphisms of excellent Dedekind schemes and $\mathscr{G}_{\mathrm{m}}$ denotes the Néron lft-model of $\Gm$ (the base ring being clear from context), the map of Néron lft-models $\Res_{\mathcal{B}/S} \mathscr{G}_{\mathrm{m}} \to \Res_{\mathcal{A}/S} \mathscr{G}_{\mathrm{m}}$ is a closed immersion and its quotient is the Néron lft-model of its generic fibre \cite[Lemmata 2.8, 2.9, and proofs thereof]{Ov}.
\item If $S$ is the spectrum of a discrete valuation ring and $\mathscr{F}$ is an fppf-sheaf of Abelian groups on $S,$ then there exists a morphism $\mathscr{F}\to \mathscr{F}\sep$ such that for any flat $Z\to S,$ any generically trivial element of $\mathscr{F}\sep(Z)$ is trivial and such that for any \it separated \rm Abelian group algebraic space $\mathscr{G}$ over $S,$ any morphism $\mathscr{F} \to \mathscr{G}$ factors uniquely through $\mathscr{F}\sep.$ In fact, $\mathscr{F}\sep$ is the quotient of $\mathscr{F}$ by the schematic closure of the unit section (defined functorially; see \cite[ p. 40]{RaynaudPic}).  If $\mathscr{F}$ is representable, then this functorial schematic closure of the unit section coincides with the classical one (\it ibid.\rm). 
\item An affine scheme $\Spec R$ is said to be \it seminormal \rm if for all $x, y \in R$ such that $y^2=x^3,$ there is a (necessarily unique) $a\in R$ such that $x=a^2$ and $y=a^3.$ A scheme $X$ is seminormal if it can be covered by seminormal open affine subschemes \cite[Tags 0EUL and 0EUN]{Stacks}. See \cite[Tag 0EUK]{Stacks} and \cite[subsection 2.4.1]{Ov} for more details.
\item For a closed immersion $Z\to X$ and a morphism $Z\to T,$ of schemes, we denote the push-out of the resulting diagram in the category of schemes (if it exists) by $X\cup_ZT.$ See \cite[subsection 2.4.2]{Ov} for general results regarding their existence and behaviour (especially under arbitrary base change), and for further references.  
\item For a proper, flat, and finitely presented morphism $X\to S$ of schemes, we denote by $\Pic_{X/S}$ the fppf-sheafification of the presheaf \begin{align}T\mapsto \Pic(X\times_ST)\label{PicPresh}\end{align} (cf. \cite[Chapter 8.1, Definition 2]{BLR}). In fact, $\Pic_{X/S}$ is also equal to the \it étale \rm sheafification of the presheaf (\ref{PicPresh}). This is well-known (\cite[p. 28, (1.2)]{RaynaudPic} or \cite[p. 203]{BLR}); for a complete proof, see \cite[Proposition 2.27]{Ov}. 
\end{itemize}

\section{Chai's conjecture for Jacobians}
\subsection{Degrees and compatible rigidification}
Let $f\colon X\to S$ be a morphism of schemes which is proper, flat, and locally of finite presentation. A \it rigidificator \rm of $\Pic_{X/S}$ (cf. \cite{BLR, RaynaudPic}) is a closed subscheme $Y\subseteq X$ finite and flat over $S$ such that for all $S$-schemes $T,$ the canonical map $\Gamma(X_T, \Og_{X_T}) \to \Gamma(Y_T, \Og_{Y_T})$ is injective. We have an associated \it rigidified Picard functor \rm $(\Pic_{X/S}, Y)$ \cite[p. 30]{RaynaudPic}, which is representable by an algebraic space locally of finite presentation over $S$ \cite[ Théorème 2.3.1]{RaynaudPic}. If $f$ is of relative dimension 1 (the only case of interest to us), then $\Pic_{X/S}$ and $(\Pic_{X/S}, Y)$ are formally smooth over $S$ \cite[Corollaire 2.3.2]{RaynaudPic}.
We denote the functors $\Res_{X/S}\Gm$ and $\Res_{Y/S}\Gm$ by $V_X^\times$ and $V_Y^\times;$ both are representable by group schemes over $S$ and we have a canonical exact sequence
$$0\to V_X^\times \to V_Y^\times \to (\Pic_{X/S},Y)\to \Pic_{X/S} \to 0$$ of sheaves on the big étale site of $S$ \cite[Propositions 2.1.2 and 2.4.1]{RaynaudPic}. \\ 
If $S$ is a Dedekind scheme and $X/S$ is a proper and flat relative curve, we construct open and closed subfunctors $P_{X/S}$ (resp. $(P_{X/S}, Y)$) of $\Pic_{X/S}$ (resp. $(\Pic_{X/S}, Y)$) as follows: Denote by $K$ the field of fractions of $S$ and let $K'$ be a finite separable extension of $K$ such that all irreducible components of $X_{K'}$ are geometrically irreducible. Let $S'$ be the integral closure of $S$ in $K'$; the map $S'\to S$ is then finite and locally free. Let $Y_1,..., Y_d$ be the irreducible components of $X_{K'}$ (with their natural scheme structure) and let $X_1,..., X_d$ be their scheme-theoretic closures in $X_{S'}.$ The schemes $X_1,..., X_d$ are proper and flat over $S'$ with generic fibres $Y_1,...,Y_d.$ We consider the degree function 
$$\deg_0\colon \Pic_{X\times_SS'/S'} \to \Z^d$$ which is induced by the map $\mathscr{L} \mapsto (\deg \mathscr{L}\mid_{X_1\times_{S'} T}, ..., \deg \mathscr{L}\mid_{X_d\times_{S'} T})$ for an $S'$-scheme $T$ and $\mathscr{L}\in \Pic (X\times_ST).$ The degree function induces a map 
$$\deg \colon \Pic_{X/S} \to \Res_{S'/S} \Z^d.$$ It is easy to see that $\Res_{S'/S} \Z^d$ is separated and étale over $S,$ so the subfunctor
$$P_{X/S}:=\ker \deg$$ is open and closed in $\Pic_{X/S}.$ By \cite[Chapter 9.2, Corollary 14]{BLR}, the generic fibre of $\ker \deg_0$ is equal\footnote{After base change to $K\alg,$ the map $\deg_0 \colon\Pic_{X_{K\alg}/K\alg} \to \Z^d$ is the composition of the degree map $\Pic_{X_{K\alg}/K\alg} \to \Z^d$ from \it loc. cit. \rm with the map $\Z^d\to \Z^d$ given by the matrix $\mathrm{diag}(m_1,..., m_d),$ where $m_j$ is the multiplicity of $X_{j,K\alg}$ for $j=1,..., d.$} to $\Pic^0_{X_{K'}/K'}.$ Hence the generic fibre of $P_{X/S}$ is equal to $\Pic^0_{X_K/K}.$ We construct $(P_{X/S},Y)$ in a completely analogous way; both constructions do not depend upon the initial choice of $K'.$ In particular, we obtain an exact sequence
$$0\to V_X^\times \to V_Y^\times \to (P_{X/S}, Y) \to P_{X/S} \to 0.$$
\begin{lemma} 
Let $S'\to S$ be a morphism of schemes which is finite, of finite presentation, and flat. Let $X'\to S'$ be proper, flat, and of finite presentation. Let $Y$ be a rigidificator of $\Pic_{X'/S'}.$ Then $Y$ is a rigidificator of $\Pic_{X'/S}$ and we have a canonical isomorphism \label{RigPicReslem}
$$(\Pic_{X'/S},Y)=\Res_{S'/S} (\Pic_{X'/S'}, Y).$$ 
\end{lemma}
\begin{proof}
Let $T$ be any $S$-scheme. Then $X'\times_{S}T=X'\times_{S'}(S'\times_{S} T),$ and a similar identity holds for $Y.$ This shows that the map $\Gamma(X'\times_{S}T, \Og_{X'\times_{S}T}) \to \Gamma(Y\times_{S}T, \Og_{Y\times_{S}T})$ is injective. The proof of the second claim is purely formal and left to the reader. 
\end{proof}\\
Now let $R\subseteq R'$ be a finite flat extension of rings such that $R$ is an Henselian discrete valuation ring and $R'$ is a finite product of Henselian discrete valuation rings. Put $S:=\Spec R,$ $S':=\Spec R'.$ Suppose we have a commutative diagram 
$$\begin{CD}
X'@>{\psi}>> X\\
@VVV@VVV\\
S'@>>> S
\end{CD}$$ 
such that the vertical arrows are proper and flat (hence automatically of finite presentation), and such that the special fibres of $X'\to S'$ have no embedded components. Moreover, assume that $\psi$ is $S$-universally scheme-theoretically dominant and that there exists an open subset $\mathscr{U}\subseteq X$ such that $\psi^{-1}(\mathscr{U})$ is fibre-wise dense in $X'$ and such that $\psi$ is an isomorphism above $\mathscr{U}.$ Then we have the following
\begin{proposition}
There exists a rigidificator $Y\subseteq X$ of $\Pic_{X/S}$ which is contained in $\mathscr{U}$ such that $\psi^{-1}(Y)$ is a rigidificator of $\Pic_{X'/S'}.$ \label{comprigexprop}
\end{proposition}

\begin{proof}
We shall identify $\mathscr{U}$ and $\psi^{-1}({\mathscr{U}}).$ Let $Z_1,$ ..., $Z_n$ be the (reduced) irreducible components of the special fibre of $X'\to S.$ By our assumption on $\mathscr{U},$ the intersection $\mathscr{U}\cap Z_j$ is non-empty for all $j=1,..., n.$ For each $j,$ we choose a closed point $x_j\in \mathscr{U}\cap Z_j$ with image $\mathfrak{m}_j$ in $S'$ which does not lie on the intersection of two irreducible components and such that $\Og_{(X'\times_{R'} \Spec R'/\mathfrak{m_j}), x_j}$ is Cohen-Macaulay for all $j=1,..., n$ \cite[Exposé VI\textsubscript{A}, Lemme 1.1.2]{SGA3}. Raynaud \cite[proof of Proposition 2.2.3(b)]{RaynaudPic} has shown that there exists a rigidificator $Y$ of $\Pic_{X'/S'}$ whose special fibre is set-theoretically equal to $\{x_1,..., x_n\}.$ Because $Y$ is a semilocal scheme, we see that $Y\subseteq \mathscr{U},$ so $Y$ can be regarded as a closed subscheme of $X.$ By the preceding Lemma, we know that $Y$ is a rigidificator of $\Pic_{X'/S}.$ Because $\psi$ is $S$-universally scheme-theoretically dominant, $Y$ is a rigidificator of $\Pic_{X/S}$ as well, as desired. 
\end{proof}\\
By abuse of notation, both $Y$ and $\psi^{-1}(Y)$ will be denoted by $Y,$ which we shall call a \it compatible rigidificator. \rm

\subsection{Semi-factorial models and Picard functors}

Let $\Og_K$ be a complete discrete valuation ring with algebraically closed residue field $k$ and field of fractions $K$ of characteristic $p>0.$ Throughout this section, $C$ will be a reduced proper curve over $K$ with normalisation $\widetilde{C}.$ A scheme $X\to \Spec \Og_K$ will be called \it semi-factorial \rm (cf. \cite[Définition 1.1]{Pép}) if the map $\Pic X \to \Pic X_K$ is surjective.

\begin{definition} A \rm model pair \it of $C$ is a morphism $\psi\colon \widetilde{\mathscr{C}} \to \mathscr{C}$ of schemes over $\Og_K$ such that the following conditions are satisfied: \\
(i) The structure morphisms $\widetilde{f}\colon \widetilde{\mathscr{C}} \to \Spec \Og_K$ and $f\colon \mathscr{C}\to \Spec \Og_K$ are proper and flat,\\
(ii) $\widetilde{\mathscr{C}}$ and $\mathscr{C}$ are models of $\widetilde{C}$ and $C,$ respectively\footnote{i. e. $\mathscr{C}\times_{\Og_K} \Spec K \cong C$ and similarly for $\widetilde{\mathscr{C}}$ and $\widetilde{C}.$ We fix a choice of such an isomorphism.} , and $\psi_K$ is the normalisation morphism,\\
(iii) the scheme $\widetilde{\mathscr{C}}$ is regular,\\
(iv) the morphism $\psi$ is finite, $\Og_K$-universally scheme-theoretically dominant,  and an isomorphism away from a closed subset of $\mathscr{C}$ which is finite and flat over $\Og_K.$\\
Finally, $\psi\colon \widetilde{\mathscr{C}} \to \mathscr{C} $ will be called a \rm strong model pair \it if there exist finite, faithfully flat and regular $\Og_K$-algebras $\mathcal{B} \subseteq \mathcal{A}$ and closed immersions $\Spec \mathcal{A} \to \widetilde{\mathscr{C}}$ and $\Spec \mathcal{B} \to \mathscr{C}$ such that $\psi$ restricts to $\mathcal{B} \to \mathcal{A}$ and such that the diagram 
$$\begin{CD}
\widetilde{\mathscr{C}} @>{\psi}>> \mathscr{C} \\
@AAA@AA{\iota}A\\
\Spec \mathcal{A} @>>> \Spec \mathcal{B},
\end{CD}$$ 
is co-Cartesian.
\end{definition}
\begin{proposition}
Let $\psi\colon \widetilde{\mathscr{C}}\to \mathscr{C}$ be a strong model pair over $\Og_K$ of a proper reduced curve $C$ over $K.$ Then there is an exact sequence
\begin{align*}0 \to V_{\mathscr{C}}^\times  \to V_{\widetilde{\mathscr{C}}}^\times  \to \Res_{\mathcal{A}/\Og_K} \Gm / (\Res_{\mathcal{B}/\Og_K} \Gm) \to \Pic_{\mathscr{C}/\Og_K}\to \Pic_{\widetilde{\mathscr{C}}/\Og_K} \to 0 \end{align*}
on the big étale site of $\Spec \Og_K.$ Moreover, the Abelian sheaves $V_{\mathscr{C}}^\times$ and $V_{\widetilde{\mathscr{C}}}^\times $ are representable by group schemes of finite presentation over $\Og_K.$ \label{Otherpaperprop}
\end{proposition}
\begin{proof}
The exactness of the sequence follows from \cite[Proposition 2.30]{Ov}; the conditions (i),..., (v) as well as the standing assumption from \it loc. cit. \rm can be verified as in the proof of \cite[Corollary 3.7]{Ov}. The remaining claims follow from \cite[Chapter 8.1, Corollary 8]{BLR} (with the notation from \it loc. cit., \rm the scheme $V$ is of finite presentation over $\Og_K$ because so is the $\Og_K$-module $\mathscr{D}$; see \cite[Chapter 8.1, Theorem 7]{BLR} and \cite[Tag 00DO]{Stacks}).
\end{proof}\\
The main result in this section will be the following generalisation of a theorem of Raynaud \cite[Theorem 3.1]{LLR}: 
\begin{theorem}
Let $C$ be a reduced proper curve over $K$ which admits a strong model pair $\psi\colon \widetilde{\mathscr{C}}\to \mathscr{C}$ over $\Og_K.$ Assume moreover that $\mathscr{C}$ is semi-factorial. Then the kernel and cokernel of the natural morphism  \label{samelengththm}
$$\lambda \colon H^1(\mathscr{C}, \Og_{\mathscr{C}}) \to \Lie P^{\mathrm{sep}}_{\mathscr{C}/\Og_K}$$ have the same length. 
\end{theorem}
The proof will proceed along similar lines as that of \it loc. cit. \rm 
Given a strong model pair $\psi\colon \widetilde{\mathscr{C}} \to \mathscr{C}$ of $C,$ a compatible rigidificator $Y$ is a closed subscheme $Y\subseteq {\mathscr{C}}$ finite and flat over $\Og_K$ which is contained in an open subset of $\mathscr{C}$ above which $\psi$ is an isomorphism, such that $Y$ is a rigidificator of $\mathscr{C}$ (with respect to $\Og_K$), and a rigidificator of $\widetilde{\mathscr{C}}$ (with respect to $\Gamma(\widetilde{\mathscr{C}}, \Og_{\widetilde{\mathscr{C}}}),$ and hence also with respect to $\Og_K$). Such an object exists by Proposition \ref{comprigexprop}. Indeed, the special fibres of $\widetilde{\mathscr{C}} \to \Spec \Gamma(\widetilde{\mathscr{C}}, \Og_{\widetilde{\mathscr{C}}})$ are Cohen-Macaulay \cite[Tag 02JN]{Stacks} and therefore have no embedded components \cite[Tag 0BXG]{Stacks}. Moreover, the morphism $\psi$ is $\Og_K$-universally scheme-theoretically dominant because taking the push-out along $\Spec \mathcal{A} \to \Spec \mathcal{B}$ commutes with arbitrary base change \cite[Proposition 2.19]{Ov}. 
We obtain a canonical morphism
$$(P_{\mathscr{C}/\Og_K}, Y) \to (P_{\widetilde{\mathscr{C}}/\Og_K}, Y).$$ Now recall that we have a canonical exact sequence
\begin{align}0 \to V_{\mathscr{C}}^\times \to V_Y^\times \to (P_{\mathscr{C}/\Og_K}, Y) \to P_{\mathscr{C}/\Og_K} \to 0\label{CanonicalExactSequence}\end{align} in the étale topology, and similarly for $\widetilde{\mathscr{C}}$ \cite[Propositions 2.1.2 and 2.4.1]{RaynaudPic}.
\begin{lemma}
Let $Y$ be a compatible rigidificator of the strong model pair $\psi\colon \widetilde{\mathscr{C}} \to \mathscr{C}.$ \label{Rigexseqlem} Then we have a canonical exact sequence
$$0\to \Res_{\mathcal{A}/\Og_K} \Gm /  \Res_{\mathcal{B}/\Og_K} \Gm \to (P_{\mathscr{C}/\Og_K}, Y) \to (P_{\widetilde{\mathscr{C}}/\Og_K}, Y) \to 0.$$
\end{lemma}
\begin{proof}
Put $\mathscr{R}:=\Res_{\mathcal{A}/\Og_K} \Gm /  \Res_{\mathcal{B}/\Og_K} \Gm$ and consider the commutative diagram
$$\begin{CD}
&&0&&0&&0\\
&&@VVV@VVV@VVV\\
0@>>> V_{\widetilde{\mathscr{C}}}^\times/V_{\mathscr{C}}^\times @>>> \mathscr{R} @>>> \mathscr{R}/(V_{\widetilde{\mathscr{C}}}^\times/V_{\mathscr{C}}^\times)@>>> 0 \\
&&@VVV@VVV@VVV\\
0@>>>V_Y^\times/V_{\mathscr{C}}^\times @>>> (P_{\mathscr{C}/\Og_K}, Y) @>>> P_{\mathscr{C}/\Og_K} @>>> 0\\
&&@VVV@VVV@VVV \\
0@>>>V_Y^\times/V_{\widetilde{\mathscr{C}}}^\times @>>> (P_{\widetilde{\mathscr{C}}/\Og_K}, Y) @>>> P_{\widetilde{\mathscr{C}}/\Og_K} @>>> 0 \\
&&@VVV@VVV@VVV\\
&&0&&0&&0.
\end{CD}$$
The rows in this diagram are exact by the exact sequence (\ref{CanonicalExactSequence}) (the top one being obviously exact); the left vertical column is again clearly exact while the right vertical column is exact by Proposition \ref{Otherpaperprop}. Moreover, the central vertical column is a complex, which implies that it is exact as well. 
\end{proof}
\begin{lemma}
The algebraic space $(P_{\mathscr{C}/\Og_K}, Y)^0$ is a separated scheme over $\Og_K.$ 
\end{lemma}
\begin{proof}
If $\mathscr{C}=\widetilde{\mathscr{C}}$ is regular, this follows from \cite[Proposition 3.2]{LLR} together with Lemma \ref{RigPicReslem}. Consider the exact sequence
$$0 \to \Res_{\mathcal{A}/\Og_K} \Gm /  \Res_{\mathcal{B}/\Og_K} \Gm \to (P_{\mathscr{C}/\Og_K}, Y) \to (P_{\widetilde{\mathscr{C}}/\Og_K}, Y) \to 0$$ from Lemma \ref{Rigexseqlem}, which induces an exact sequence
$$0 \to \Res_{\mathcal{A}/\Og_K} \Gm /  \Res_{\mathcal{B}/\Og_K} \Gm \to \mathscr{G} \to (P_{\widetilde{\mathscr{C}}/\Og_K}, Y)^0 \to 0.$$
Then $\mathscr{G}$ is a group space over $\Og_K,$ and because $(P_{\widetilde{\mathscr{C}}}, Y)^0$ is separated over $\Og_K$ \cite[Proposition 3.2]{LLR}, so is $\mathscr{G}.$ In particular, $\mathscr{G}$ is a scheme \cite[Théorème 4.B]{An}. Clearly, $(P_{\mathscr{C}}, Y)^0$ is an open subspace of $\mathscr{G}$ and hence a scheme as well. 
\end{proof}
\begin{lemma}
Let $T$ be the Abelian group of all line bundles $\mathscr{L}$ on $\mathscr{C}$ (up to isomorphism) such that both $\mathscr{L}\mid_{C}$ and $\psi^\ast \mathscr{L}$ are trivial. Then $T$ is finitely generated. \label{Tfingenlem}
\end{lemma}
\begin{proof}
Consider the commutative diagram
$$\begin{CD}
0@>>>\Gamma(\widetilde{\mathscr{C}}, \Og_{\widetilde{\mathscr{C}}})^\times /\Gamma(\mathscr{C}, \Og_{\mathscr{C}})^\times@>>> \mathcal{A}^\times/\mathcal{B}^\times @>>> Q_1 @>>> 0\\
&&@V{q_C}VV@V{q}VV@VV{w}V\\
0@>>>\Gamma(\widetilde{{C}}, \Og_{\widetilde{{C}}})^\times/\Gamma({C}, \Og_{{C}})^\times @>>> {A}^\times/{B}^\times @>>> Q_2 @>>> 0,
\end{CD}$$
where the $Q_j$ are the obvious cokernels. Because $\Gamma(\mathscr{C}, \Og_{\mathscr{C}}),$ $\Gamma(\widetilde{\mathscr{C}}, \Og_{\widetilde{\mathscr{C}}}),$ $\mathcal{A},$ and $\mathcal{B}$ are finite products of discrete valuation rings, and the maps $\Gamma(\mathscr{C}, \Og_{\mathscr{C}}) \to  \Gamma(\widetilde{\mathscr{C}}, \Og_{\widetilde{\mathscr{C}}})$ and $\mathcal{B} \to \mathcal{A}$ are finite and faithfully flat, it is easy to see that $q$ is injective and that $\mathrm{coker}\, q_C$ is finitely generated. Using the exact sequence from Proposition \ref{Otherpaperprop}, we see that $T=\ker w.$ Hence the snake lemma gives an injection $T\to \mathrm{coker} \, q_C.$
\end{proof} \\
Let $\mathscr{H}_{\mathscr{C}}$ and $\mathscr{H}_{\widetilde{\mathscr{C}}}$ be the scheme-theoretic closures of the kernels of $(P_{C/K}, Y_K) \to P_{C/K}$ and $(P_{\widetilde{C}/K}, Y_K) \to P_{\widetilde{C}/K}$ in $(P_{\mathscr{C}/\Og_K}, Y)$ and $(P_{\widetilde{\mathscr{C}}/\Og_K}, Y),$ respectively. We have canonical morphisms $$V_Y^\times/\Res_{\Gamma(\mathscr{C}, \Og_{\mathscr{C}})/\Og_K} \Gm \to \mathscr{H}_{\mathscr{C}}^1$$ and $$V_Y^\times/\Res_{\Gamma(\widetilde{\mathscr{C}}, \Og_{\widetilde{\mathscr{C}}})/\Og_K}\Gm \to \mathscr{H}_{\widetilde{\mathscr{C}}}^1,$$ where
$$\mathscr{H}^1_{\mathscr{C}}:=\mathscr{H}_{\mathscr{C}}\cap (P_{\mathscr{C}/\Og_K},Y)^0,$$ and similarly for $\widetilde{\mathscr{C}}.$ Because $V_Y^\times$ is smooth, these morphisms factor through the smoothenings $\widetilde{\mathscr{H}}_{\mathscr{C}}$ and $\widetilde{\mathscr{H}}_{\widetilde{\mathscr{C}}}$ of $\mathscr{H}_{\mathscr{C}}^1$ and $\mathscr{H}_{\widetilde{\mathscr{C}}}^1,$ respectively \cite[ Chapter 7.1, Theorem 5]{BLR}. Then we have
\begin{lemma}
The maps $V_Y^\times/\Res_{\Gamma(\mathscr{C}, \Og_{\mathscr{C}})/\Og_K}\Gm \to \widetilde{\mathscr{H}}_{\mathscr{C}}$ and $V_Y^\times/\Res_{\Gamma(\widetilde{\mathscr{C}}, \Og_{\widetilde{\mathscr{C}}})/\Og_K}\Gm\to \widetilde{\mathscr{H}}_{\widetilde{\mathscr{C}}}$ are open immersions. \label{openimmlem}
\end{lemma}
\begin{proof}
The map $V_Y^\times/\Res_{\Gamma(\widetilde{\mathscr{C}}, \Og_{\widetilde{\mathscr{C}}})/\Og_K}\Gm \to \widetilde{\mathscr{H}}_{\widetilde{\mathscr{C}}}$ is an open immersion; this follows from \cite[Lemma 3.5(d)]{LLR} together with Lemma \ref{RigPicReslem}.\\
We shall first show the auxiliary claim that the cokernel of the map 
$$(V_Y^\times/\Res_{\Gamma({\mathscr{C}}, \Og_{{\mathscr{C}}})/\Og_K}\Gm)(\Og_K) \to \widetilde{\mathscr{H}}_{\mathscr{C}}(\Og_K)=\mathscr{H}^1_{\mathscr{C}}(\Og_K)$$ is finitely generated. The cokernel $Q$ of $(V_Y^\times/\Res_{\Gamma(\widetilde{\mathscr{C}}, \Og_{\widetilde{\mathscr{C}}})/\Og_K}\Gm )(\Og_K)\to \widetilde{\mathscr{H}}_{\widetilde{\mathscr{C}}}(\Og_K)$ is a finite Abelian group. Now consider the map
$\mathscr{H}^1_{\mathscr{C}}(\Og_K) \to Q$ induced by the composition 
$$\mathscr{H}^1_{\mathscr{C}}(\Og_K) \to \mathscr{H}^1_{\widetilde{\mathscr{C}}}(\Og_K) = \widetilde{\mathscr{H}}_{\widetilde{\mathscr{C}}}(\Og_K) \to Q$$ and $J$ be its kernel. It suffices to show that the cokernel of the induced map
$$(V_Y^\times/\Res_{\Gamma({\mathscr{C}}, \Og_{{\mathscr{C}}})/\Og_K}\Gm)(\Og_K) \to J$$ is finitely generated. An element of $J$ is a rigidified line bundle $(\mathscr{L}, \alpha)$ on $\mathscr{C}$ such that $\mathscr{L} \in T$ (Lemma \ref{Tfingenlem}). An element of the kernel of the induced map $J\to T$ clearly comes from an element of $(V_Y^\times/\Res_{\Gamma({\mathscr{C}}, \Og_{{\mathscr{C}}})/\Og_K}\Gm)(\Og_K),$ so the auxiliary claim follows.
The remainder of the argument now works as in the proof of \cite[Lemma 3.5(d)]{LLR}: what we have just shown implies that the cokernel of the induced map
$$(V_Y^\times/\Res_{\Gamma({\mathscr{C}}, \Og_{{\mathscr{C}}})/\Og_K}\Gm) \times_{\Og_K}\Spec k \to \widetilde{\mathscr{H}}_{\mathscr{C}}\times_{\Og_K} \Spec k$$
has a finitely generated group of $k$-points. But since $k$ is algebraically closed, this cokernel must be finite over $k,$ so the morphism from the Lemma is quasi-finite. Since it is an isomorphism generically and the target is regular, our claim follows from Zariski's main theorem. 
\end{proof} \\
Exactly as in \cite[Proposition 3.3]{LLR} (using the argument from \cite[p. 36]{RaynaudPic})\footnote{There is a small typographical error in \cite{LLR}; we should put $\mathscr{A}_0=\overline{\mathscr{A}}=\Og_S,$ and  $\mathscr{A}'=\Og_{S_{\epsilon}}$ in the notation of \cite{RaynaudPic} used in \it loc. cit.\rm}, we see that there is a natural exact sequence
\begin{align}0 \to \Gamma({\mathscr{C}}, \Og_{{\mathscr{C}}}) \to \Gamma(Y, \Og_Y) \to \Lie\, (P_{{\mathscr{C}}/\Og_K},Y) \to \Lie P_{{\mathscr{C}}/\Og_K} \to 0.\label{naturalseqI}\end{align}
\begin{lemma}
There is a natural diagram with exact rows
$$\begin{CD}
0@>>> \Lie \widetilde{\mathscr{H}}_{\mathscr{C}} @>>> \Lie\, (P_{\mathscr{C}/\Og_K}, Y) @>>> \Lie P_{\mathscr{C}/\Og_K} @>>> 0\\
&&@V{\omega}VV@VV{=}V@VV{\lambda}V\\
0@>>> \Lie \mathscr{H}_{\mathscr{C}} @>>>\Lie\, (P_{\mathscr{C}/\Og_K}, Y) @>>{\pi}> \Lie P_{\mathscr{C}/\Og_K}^{\mathrm{sep}}.
\end{CD}$$
\end{lemma}
\begin{proof}
The morphism denoted by $\lambda$ above coincides with the morphism $\lambda\colon H^1(\mathscr{C}, \Og_{\mathscr{C}})\to \Lie P_{\mathscr{C}/\Og_K}^{\mathrm{sep}}$ modulo the identification $\theta_{\mathscr{C}}\colon \Lie P_{\mathscr{C}/\Og_K}\to H^1(\mathscr{C}, \Og_{\mathscr{C}})$ from \cite[Proposition 1.3(b)]{LLR}. The Lemma can now be proven exactly as in \cite[Lemma 3.5(e)]{LLR}: Lemma \ref{openimmlem} shows that the map 
$$\Gamma(Y, \Og_Y)/\Gamma(\mathscr{C}, \Og_{\mathscr{C}}) \to \Lie \widetilde{\mathscr{H}}_{\mathscr{C}}$$ is an isomorphism; the top row of the diagram is induced by the exact sequence (\ref{naturalseqI}) above. Moreover, we note that the canonical morphism $(P_{\mathscr{C}/\Og_K}, Y)/\mathscr{H}_{\mathscr{C}} \to P_{\mathscr{C}/\Og_K}^{\mathrm{sep}}$ is an isomorphism (cf. \cite[Chapter 9.5, Proposition 3, \it 3rd case\rm]{BLR}), which induces the bottom row.
\end{proof}\\
We can now give the \it proof of Theorem \ref{samelengththm}, \rm which follows the argument from \cite[p. 480]{LLR}: Because $\Og_K$ is strictly Henselian, any element of $P_{\mathscr{C}/\Og_K}(\Og_K)$ is given by a line bundle $\mathscr{L}$ on $\mathscr{C}.$ Because $Y$ is semilocal, $\mathscr{L}\!\mid_Y$ is trivial, so $(P_{\mathscr{C}/\Og_K},Y)(\Og_K) \to P_{\mathscr{C}/\Og_K}(\Og_K)$ is surjective.  Moreover, the map $P_{\mathscr{C}/\Og_K}(\Og_K)\to P_{\mathscr{C}/\Og_K}^{\mathrm{sep}}(\Og_K)$ is easily seen to be surjective since $\mathscr{C}$ is semi-factorial. By \cite[Chapter 9.6, Lemma 2]{BLR}, the induced map $$(P_{\mathscr{C}/\Og_K},Y)^0 (\Og_K) \to P_{\mathscr{C}/\Og_K}^{\mathrm{sep},0}(\Og_K)$$ is surjective as well. In particular, using \cite[Theorem 2.1(a)]{LLR}, we see that 
\begin{align*}
\ell_{\Og_K}(\ker \lambda) = \ell_{\Og_K}(\mathrm{coker} \, \omega) = \ell_{\Og_K} (\mathrm{coker} \, \pi) = \ell_{\Og_K}(\mathrm{coker} \, \lambda),
\end{align*}
as desired. \hfill \qed
\subsection{Relationship with the Néron model}
We keep the notation from above. In particular, $C$ is a proper reduced curve over $K.$ We have the following
\begin{proposition}
Let $G$ be a smooth connected group scheme over $K$ which admits a Néron lft-model $\mathscr{N}.$ Let $\mathscr{G}\to \Spec \Og_K$ be a smooth separated model of $G$ such that the map $\mathscr{G}(\Og_K) \to G(K)$ is an isomorphism and such that $\Phi_{\mathscr{G}}:=\mathscr{G}_k(k)/\mathscr{G}^0_k(k)$ is finitely generated as an Abelian group. Then the map $\mathscr{G}\to \mathscr{N}$ is an isomorphism. \label{fingennérprop}
\end{proposition}
\begin{proof}
By definition, there is a unique morphism $\mathscr{G}\to \mathscr{N}$ which extends the identity at the generic fibre. We shall first show that the induced map on identity components
$\mathscr{G}^0 \to \mathscr{N}^0$ is an isomorphism; this part of the argument is essentially taken from \cite[p. 269]{BLR}.
By assumption, the map $\mathscr{G}(\Og_K) \to \mathscr{N}(\Og_K)$ is an isomorphism; in particular, the map $\mathscr{G}(k)\to \mathscr{N}(k)$ is surjective. The cokernel of the induced morphism $$\nu \colon \mathscr{G}_k^0 \to \mathscr{N}^0 _k$$ is a smooth connected algebraic group over $k.$ The snake lemma shows that $(\mathrm{coker}\,\nu)(k)$ is a finitely generated Abelian group. But because $k$ is algebraically closed, $(\mathrm{coker}\,\nu)(k)$ is $n$-divisible for any natural number $n$ invertible in $k.$ This is only possible if $\mathrm{coker}\,\nu=0,$ so the map $\mathscr{G}^0_k \to \mathscr{N}^0_k$ is surjective. Because the fibres of $\mathscr{G}^0$ and $\mathscr{N}^0$ are of the same dimension, the map $\mathscr{G}^0\to \mathscr{N}^0$ must be quasi-finite. Since it is also an isomorphism generically, Zariski's main theorem shows that it is an isomorphism. \\
Now let $j\colon \Spec k \to \Spec \Og_K$ be the canonical closed immersion and denote by $\Phi_{\mathscr{N}}$ the Abelian group $\mathscr{N}(\Og_K)/\mathscr{N}^0(\Og_K)=\mathscr{N}_k(k)/\mathscr{N}^0_k(k).$ We can view $\Phi_{\mathscr{N}}$ and $\Phi_{\mathscr{G}}$ as group schemes over $k,$ and we have a diagram with exact rows
$$\begin{CD}
0@>>> \mathscr{G}^0@>>>\mathscr{G}@>>> j_\ast \Phi_{\mathscr{G}}@>>>0\\
&&@VVV@VVV@VVV\\
0@>>> \mathscr{N}^0 @>>> \mathscr{N} @>>>j_\ast \Phi_{\mathscr{N}}@>>>0
\end{CD}$$
(the rows are even exact in the Zariski topology). We already know that the left vertical arrow is an isomorphism. This implies that the right vertical arrow is also an isomorphism since $\mathscr{G}(\Og_K)\to \mathscr{N}(\Og_K)$ is an isomorphism by assumption, and we have just seen that it carries $\mathscr{G}^0(\Og_K)$ isomorphically into $\mathscr{N}^0(\Og_K).$ In particular, the central vertical arrow is an isomorphism as well, which implies our claim. 
\end{proof}
\begin{corollary} 
Let $C$ be a proper reduced curve over $K$ and let $\psi\colon \widetilde{\mathscr{C}} \to \mathscr{C}$ be a strong model pair of $C$ such that $\mathscr{C}$ is semi-factorial. \label{Nérisomcor} Moreover, assume that $\Pic^0_{C/K}$ admits a Néron lft-model over $\Og_K.$ Then $P_{\mathscr{C}/\Og_K}^{\mathrm{sep}}$ is the Néron lft-model of $\Pic^0_{C/K}.$ 
\end{corollary}
\begin{proof}
First note that $P_{\mathscr{C}/K}^{\mathrm{sep}}$ is representable by a smooth separated group scheme over $\Og_K;$ the proof works as in \cite[Chapter 9.5, proof of Proposition 3, \it 3rd case\rm]{BLR}. Moreover, by \cite[Chapter 9.2, Corollary 14]{BLR}, the Abelian group $$P_{\mathscr{C}/\Og_K}^{\mathrm{sep}}(k)/P_{\mathscr{C}/\Og_K}^{\mathrm{sep},0}(k)=P_{\mathscr{C}/\Og_K}^{\mathrm{sep}}(\Og_K)/P_{\mathscr{C}/\Og_K}^{\mathrm{sep},0}(\Og_K)$$ is finitely generated. Because $k$ is algebraically closed, $\Br \Gamma(C, \Og_C) =0$ \cite[Theorem 1.2.15]{CS}, so $\Pic_{C/K}(K)=\Pic C.$ Because $\mathscr{C}$ is semi-factorial, the map $\Pic_{\mathscr{C}/\Og_K}(\Og_K) \to \Pic_{C/K}(K)$ is surjective. This map induces a surjection $P_{\mathscr{C}/\Og_K}(\Og_K) \to \Pic^0_{C/K}(K)$ as $P_{\mathscr{C}/\Og_K}$ is open and closed in $\Pic_{\mathscr{C}/\Og_K}.$ Hence Proposition \ref{fingennérprop} implies the claim. 
\end{proof}
\begin{corollary}
Let $C$ be a proper reduced curve over $K$ such that $\Pic^0_{C/K}$ admits a Néron lft-model $\mathscr{N}$ over $\Og_K.$ Let $\psi\colon \widetilde{\mathscr{C}}\to \mathscr{C}$ be a strong model pair of $C$ and suppose that $\mathscr{C}$ is semi-factorial. Moreover, let $\widetilde{\mathscr{N}}$ be the Néron lft-model of $\Pic^0_{\widetilde{C}/K}.$ Then there is a commutative diagram \label{lengthscor}
$$\begin{CD}
H^1(\mathscr{C}, \Og_{\mathscr{C}}) @>{\psi^\ast}>> H^1(\widetilde{\mathscr{C}}, \Og_{\mathscr{\widetilde{\mathscr{C}}}})\\
@V{\lambda}VV@VV{\widetilde{\lambda}}V\\
\Lie \mathscr{N}@>>> \Lie \widetilde{\mathscr{N}}
\end{CD}$$
such that $\ell_{\Og_K}(\ker \lambda)=\ell_{\Og_K}(\mathrm{coker}\, \lambda)$ and $\ell_{\Og_K}(\ker \widetilde{\lambda})=\ell_{\Og_K}(\mathrm{coker}\, \widetilde{\lambda}).$
\end{corollary}
\begin{proof}
Commutativity of the diagram is clear. The claim on lengths follows from Theorem \ref{samelengththm} together with Corollary \ref{Nérisomcor} (for $\widetilde{\lambda},$ this also follows from \cite[Theorem 3.1]{LLR}).
\end{proof}\\
The requirement that $\Pic^0_{C/K}$ be semiabelian has the following consequence: 
\begin{proposition}
Let $C$ be a reduced proper curve over $K$ and assume that $\Pic^0_{C/K}$ is semiabelian. Then $C$ is seminormal. Moreover, the irreducible components of $\widetilde{C}$ are either smooth over $K$ or of arithmetic genus 0. In particular, $\Pic^0_{\widetilde{C}/K}$ is an Abelian variety. \label{seminormalpropII}
\end{proposition}
\begin{proof}
We may suppose without loss of generality that $C$ is connected. Assume that $\Pic^0_{C/K}$ is semiabelian. Let $C^{\mathrm{sn}}$ be the seminormalisation of $C,$ such that we have a factorisation $\widetilde{C} \overset{\nu}\to C^{\mathrm{sn}} \overset{\varsigma}{\to} C.$ If $\varsigma$ is not an isomorphism, \cite[Theorem 2.24(i)]{Ov} shows that we have a factorisation 
$C^{\mathrm{sn}}=C_1 \to ... \to C_n=C$ such that for each $i=1,..., n-1,$ we have a closed point $x_{i+1}$ of $C_{i+1}$ and a closed immersion $\kappa(x_{i+1})[\epsilon]/\langle \epsilon^2 \rangle \to C_i$ such that the diagram

$$\begin{CD}
C_{i}@>>> C_{i+1}\\
@AAA@AAA \\
\Spec \kappa(x_{i+1})[\epsilon]/\langle \epsilon ^2 \rangle @>>> \Spec \kappa(x_{i+1})
\end{CD}$$
is co-Cartesian.
Because $C$ is connected, so are $C_i,$ and $C_{i+1},$ and the $K$-algebras $\Gamma(C_i, \Og_{C_i})$ as well as $\Gamma(C_{i+1}, \Og_{C_{i+1}})$ are fields contained in $\kappa(x_{i+1}).$ Consider the exact sequence
\begin{align*}0 &\to \Res_{\Gamma(C_{i+1}, \Og_{C_{i+1}})/K} \Gm \to \Res_{\Gamma(C_{i}, \Og_{C_{i}})/K} \Gm  \to \Res_{\kappa(x_{i+1})/K} \Ga \\& \to \Pic^0_{C_i/K} \to \Pic^0_{C_{i+1}/K} \to 0\end{align*} \cite[Proposition 2.30]{Ov}. Counting dimensions and using that quotients of split unipotent groups are split unipotent \cite[Theorem B.3.4]{CGP}, we see that $\Pic^0_{C/K}$ has a non-trivial split unipotent subquotient, which is impossible. In particular, $C$ is seminormal. By \cite[Proposition 2.30]{Ov}, we see that the map $\Pic^0_{C/K} \to \Pic^0_{\widetilde{C}/K}$ is surjective, so $\Pic^0_{\widetilde{C}/K}$ is semiabelian. Hence the same is true for $\Pic^0_{C'/K},$ where $C'$ is an irreducible component of $\widetilde{C}.$ By Proposition \ref{Steinredprop} together with the fact that $\delta(K)=1$ since $k$ is perfect (Lemma \ref{plusonelem}), we conclude that $C'$ is geometrically integral over the field $\Gamma(C', \Og_{C'}).$ After replacing $\Gamma(C', \Og_{C'})$ by a finite separable extension if necessary, \cite[ Chapter 9.2, Proposition 4]{BLR} shows that $\Pic^0_{C'/\Gamma(C', \Og_{C'})}$ contains no algebraic torus. In particular, $\Pic^0_{C'/K}=\Res_{\Gamma(C', \Og_{C'})/K} \Pic^0_{C'/\Gamma(C', \Og_{C'})}$ contains no torus and is therefore an Abelian variety. Hence $\Gamma(C', \Og_{C'})$ must be separable over $K$ or we must have $\Pic^0_{C'/\Gamma(C', \Og_{C'})}=0$ \cite[Example A.5.6]{CGP}, which proves the remaining claim.
\end{proof}
\subsection{Proof of Chai's conjecture for Jacobians} \label{Chaiproofsubsec}
As before, we let $\Og_K$ be a complete discrete valuation ring with field of fractions $K$ and algebraically closed residue field $k.$ Let $C$ be a proper reduced curve over $K$ such that $G:=\Pic^0_{C/K}$ is a semiabelian variety with toric part $T$ and Abelian part $E.$ Choose a finite separable extension $L$ of $K$ such that $G$ has semiabelian reduction over the integral closure $\Og_L$ of $\Og_K$ in $L.$ Let $\mathscr{T},$ $\mathscr{N},$ and $\widetilde{\mathscr{N}}$ be the Néron lft-models of $T,$ $G,$ and $E,$ respectively, and let $ \mathscr{T}_L,$ $\mathscr{N}_L,$ and $\widetilde{\mathscr{N}}_L$ be the Néron lft-models of the base changes to $L.$ Following Chai \cite{Chai}, we define the \it base change conductor \rm of $G$ as
$$c(G):=\frac{1}{[L:K]} \ell_{\Og_L}(\mathrm{coker}((\Lie \mathscr{N})\otimes_{\Og_K} \Og_L \to \Lie \mathscr{N}_L)),$$
and make analogous definitions for $T$ and $E.$ The base change conductor does not depend upon the choice of $L.$ We have the following
\begin{theorem} \rm (Chai's conjecture for Jacobians) \it
If $G=\Pic^0_{C/K},$ $T,$ and $E$ are as above, then \label{Chaithm}
$$c(G)=c(T)+c(E).$$
\end{theorem}
The proof will occupy the remainder of this section. Recall that $L$ is a finite separable extension of $K$ such that $\Pic^0_{C/K}$ has semiabelian reduction over $\Og_L;$ in particular, the torus $T$ splits over $L.$ We begin with the following
\begin{proposition}
Let $D$ be a proper reduced seminormal curve over $K.$ Then $D$ admits a strong model pair $\psi\colon \widetilde{\mathscr{D}} \to  \mathscr{D}$ such that $\mathscr{D}$ is semi-factorial. \label{modpairexprop}
\end{proposition}
\begin{proof}
The semi-factorial model $\mathscr{D}$ of $D$ constructed in \cite[Theorem 4.3]{Ov} is naturally endowed with a morphism $\psi\colon \widetilde{\mathscr{D}} \to  \mathscr{D},$ which is a strong model pair by construction.
\end{proof}\\
In particular, Proposition \ref{modpairexprop} provides a co-cartesian diagram 
$$\begin{CD}
\widetilde{\mathscr{D}} @>{\psi}>> \mathscr{D}\\
@AAA@AA{\iota}A\\
\Spec \mathcal{A} @>>> \Spec \mathcal{B},
\end{CD}$$
where $\mathcal{A}$ and $\mathcal{B}$ are finite products of regular finite flat extensions of $\Og_K.$ We denote their generic fibres by $A$ and $B,$ respectively.
In particular, we have an exact sequence of coherent $\Og_{\mathscr{D}}$-modules
$$0 \to \Og_{\mathscr{D}} \to \psi_\ast \Og_{\widetilde{\mathscr{D}}} \to \iota_\ast (\mathcal{A}/\mathcal{B}) \to 0.$$ Taking cohomology, we obtain an exact sequence
\begin{align}0 \to\Gamma(\mathscr{D}, \Og_{\mathscr{D}}) \to \Gamma(\widetilde{\mathscr{D}}, \Og_{\widetilde{\mathscr{D}}})\to \mathcal{A}/\mathcal{B} \to H^1(\mathscr{D}, \Og_{\mathscr{D}}) \overset{h}\to H^1(\widetilde{\mathscr{D}}, \Og_{\widetilde{\mathscr{D}}}) \to 0.\label{cohomseq}\end{align}
\begin{lemma}
Let $D$ be as in Proposition \ref{modpairexprop} and assume that $\Pic^0_{D/K}$ is semiabelian. Let $\mathscr{T}$ be the Néron lft-model of its toric part. Then there is a canonical map
$$\hbar\colon \ker h \to \Lie \mathscr{T}$$ \label{samelengthlemII}
which is generically an isomorphism and whose kernel and cokernel are of the same length.  
\end{lemma}
\begin{proof} 
The Néron lft-model of $\Gm$ will be denoted by $\NGm;$ the base ring will always be clear from context. Define
$$R_{\widetilde{D}/D} := \Res_{\Gamma(\widetilde{\mathscr{D}}, \Og_{\widetilde{\mathscr{D}}})/\Og_K}\Gm / \Res_{\Gamma({\mathscr{D}}, \Og_{{\mathscr{D}}})/\Og_K}\Gm,$$
$$\mathscr{R}_{\widetilde{D}/D} := \Res_{\Gamma(\widetilde{\mathscr{D}}, \Og_{\widetilde{\mathscr{D}}})/\Og_K}\NGm / \Res_{\Gamma({\mathscr{D}}, \Og_{{\mathscr{D}}})/\Og_K}\NGm,$$
$$R_{\mathcal{A}/\mathcal{B}} :=\Res_{\mathcal{A}/\Og_K} \Gm / \Res_{\mathcal{B}/\Og_K} \Gm,$$
and
$$\mathscr{R}_{\mathcal{A}/\mathcal{B}} :=\Res_{\mathcal{A}/\Og_K} \NGm / \Res_{\mathcal{B}/\Og_K} \NGm.$$ Note that these group schemes are smooth and separated over $\Og_K.$\\
\it Step 1: \rm Since $\mathscr{R}_{\mathcal{A}/\mathcal{B}}$ is smooth over $\Og_K,$ there is a canonical map $\mathscr{R}_{\mathcal{A}/\mathcal{B}} \to \mathscr{T},$ which induces a map $\mathcal{A}/\mathcal{B} \to \Lie \mathscr{T}$ on Lie algebras. The composition $\Res_{\Gamma(\widetilde{\mathscr{D}}, \Og_{\widetilde{\mathscr{D}}})/\Og_K}\Gm \to \mathscr{T}$ vanishes generically (and hence everywhere), so the induced map $\Gamma(\widetilde{\mathscr{D}}, \Og_{\widetilde{\mathscr{D}}}) \to \mathcal{A}/\mathcal{B} \to \Lie\mathscr{T}$ vanishes. This induces the desired map $\hbar.$ \\
\it Step 2: \rm We claim that the map $\mathscr{R}_{\mathcal{A}/\mathcal{B}} \to \mathscr{T}$ is surjective on $\Og_K$-points. Since both group schemes are the Néron lft-models of their generic fibres, it suffices to show that the map is surjective on $K$-points. But this follows from the fact that 
\begin{align*}
& H^1(K, \Res_{\Gamma(\widetilde{{D}}, \Og_{\widetilde{{D}}})/K}\Gm / \Res_{\Gamma({{D}}, \Og_{{{D}}})/K}\Gm) \\
=& \ker ( \Br ( \Gamma({{D}}, \Og_{{{D}}}) ) \to \Br(\Gamma(\widetilde{{D}}, \Og_{\widetilde{{D}}})) )\\
=&0
\end{align*}
\cite[Theorem 1.2.15]{CS}. By \cite[Chapter 9.6, Lemma 2]{BLR}, the induced map
$$R_{\mathcal{A}/\mathcal{B}}(\Og_K) \to \mathscr{T}^0(\Og_K)$$ is surjective as well. \\
\it Step 3: \rm Let $\mathscr{K}$ be the kernel of the map $\mathscr{R}_{\mathcal{A}/\mathcal{B}} \to \mathscr{T}.$ Now we claim that the induced map 
$\rho\colon\mathscr{R}_{\widetilde{D}/D} \to \mathscr{K}$ satisfies the universal property of the smoothening (see \cite[p. 174]{BLR}). This follows immediately from the fact that $ \mathscr{R}_{\widetilde{D}/D}$ is the Néron lft-model of the generic fibre of $\mathscr{K}.$ Put $$ \mathscr{R}_{\widetilde{D}/D}':= \rho^{-1}(\mathscr{K}\cap R_{\mathcal{A}/\mathcal{B}}).$$ Then $\mathscr{R}_{\widetilde{D}/D}'$ is the group smoothening of $\mathscr{K}\cap R_{\mathcal{A}/\mathcal{B}};$ this is an immediate consequence of the universal properties of the group smoothening and the Néron lft-model. Note that, moreover, we have an exact sequence
$$0 \to \mathscr{K} \cap R_{\mathcal{A}/\mathcal{B}} \to  R_{\mathcal{A}/\mathcal{B}} \to \mathscr{T}^0$$
of group schemes of finite type over $\Og_K.$ \\ 
\it Step 4: \rm Consider the commutative diagram 
$$\begin{CD}
0@>>> \Lie R_{\widetilde{D}/D} @>>> \Lie R_{\mathcal{A}/\mathcal{B}} @>>> \ker h @>>> 0\\
&&@V{\alpha}VV@VV{=}V@VV{\hbar}V\\
0@>>> \Lie \mathscr{K} @>>> \Lie R_{\mathcal{A}/\mathcal{B}} @>>{u}> \Lie \mathscr{T};
\end{CD}$$
the top row being induced by the sequence (\ref{cohomseq}) above. The remainder of the argument now works as in \cite[p. 480]{LLR}: The snake lemma shows that $\ker \hbar=\mathrm{coker} \, \alpha,$ we clearly have $\mathrm{coker}\, \hbar = \mathrm{coker}\, u,$ and \cite[Theorem 2.1(a)]{LLR} shows that 
$$\ell_{\Og_K}( \mathrm{coker}\, \alpha) = \ell_{\Og_K}( \mathrm{coker}\, u).$$ \end{proof} \\
We choose a strong model pair $\psi \colon \widetilde{\mathscr{C}}\to  \mathscr{C}$ of $C$ such that $\mathscr{C}$ is semi-factorial; this is possible by Propositions \ref{seminormalpropII} and \ref{modpairexprop}. In particular, we have closed immersions $\Spec \mathcal{A} \to \widetilde{\mathscr{C}}$ and $\Spec \mathcal{B} \to {\mathscr{C}}$ such that $\mathcal{A}$ and $\mathcal{B}$ are regular. We have already chosen a finite separable extension $L$ of $K$ such that $G=\Pic^0_{C/K}$ acquires semiabelian reduction over $\Og_L$ (in particular, $T$ splits over $L$). Let $\mathscr{T}_L,$ $\mathscr{N}_L,$ and $\widetilde{\mathscr{N}}_L$ be the Néron lft-models of $T_L,$ $G_L$, and $E_L,$ respectively. The argument given in the proof of \cite[Lemma 11.2]{CY} shows that the sequence
$$0\to \mathscr{T}_L\to\mathscr{N}_L \to \widetilde{\mathscr{N}}_L\to 0$$
is exact. \\
We have now assembled all the tools needed to give the \it proof of Theorem \ref{Chaithm}. \rm First note that we have a commutative diagram

$$\begin{CD}
0@>>> (\ker h) \underset{\Og_K}{\otimes} \Og_L @>>> H^1(\mathscr{C}, \Og_{\mathscr{C}}) \underset{\Og_K}{\otimes} \Og_L@>>> H^1(\widetilde{\mathscr{C}}, \Og_{\widetilde{\mathscr{C}}})\underset{\Og_K}{\otimes} \Og_L @>>> 0 \\
&&@V{\phi}VV@V{\chi}VV@V{\widetilde{\chi}}VV \\
0 @>>> \Lie \mathscr{T}_L @>>> \Lie \mathscr{N}_L @>>> \Lie \widetilde{\mathscr{N}}_L @>>>0
\end{CD}$$
The snake lemma now tells us that 
\begin{align}
\ell_{\Og_L}(\ker \phi) - \ell_{\Og_L}(\ker \chi) + \ell_{\Og_L}(\ker \widetilde{\chi}) - \ell_{\Og_L}(\mathrm{coker}\,\phi) + \ell_{\Og_L}(\mathrm{coker}\, \chi) - \ell_{\Og_L}(\mathrm{coker}\, \widetilde{\chi})=0.\label{zeroeq1}\end{align}
Now note that the morphisms $\phi,$ $\chi,$ and $\widetilde{\chi}$ factor as
$$\begin{CD} (\ker h) \otimes_{\Og_K} \Og_L @>{\hbar \otimes \Og_L}>> (\Lie \mathscr{T}) \otimes_{\Og_K}\Og_L @>{\omega}>> \Lie \mathscr{T}_L,\end{CD}$$
$$\begin{CD}H^1(\mathscr{C}, \Og_{\mathscr{C}}) \otimes_{\Og_K} \Og_L @>{\lambda\otimes \Og_L}>> (\Lie \mathscr{N}) \otimes_{\Og_K} \Og_L @>{\xi}>> \Lie \mathscr{N}_L,\end{CD}$$ and 
$$\begin{CD}H^1(\widetilde{\mathscr{C}}, \Og_{\widetilde{\mathscr{C}}}) \otimes_{\Og_K} \Og_L @>{\widetilde{\lambda}\otimes \Og_L}>> (\Lie \widetilde{\mathscr{N}}) \otimes_{\Og_K} \Og_L @>{\widetilde{\xi}}>> \Lie \widetilde{\mathscr{N}}_L,\end{CD}$$
respectively. 
Moreover, note that 
$$\ker (\hbar \otimes \Og_L)= (\ker h)_{\mathrm{tors}} \otimes_{\Og_K} \Og_L = \ker \phi,$$
$$\ker (\lambda\otimes\Og_L) = H^1(\mathscr{C}, \Og_{\mathscr{C}})_{\mathrm{tors}}\otimes_{\Og_K} \Og_L = \ker \chi,$$ and
$$\ker (\widetilde{\lambda}\otimes\Og_L) = H^1(\widetilde{\mathscr{C}}, \Og_{\widetilde{\mathscr{C}}})_{\mathrm{tors}}\otimes_{\Og_K} \Og_L = \ker \widetilde{\chi}.$$
Because $\Og_L$ is flat over $\Og_K,$ Corollary \ref{lengthscor} and Lemma \ref{samelengthlemII} imply that 
$\ell_{\Og_L} (\ker(\hbar\otimes \Og_L)) = \ell_{\Og_L}(\mathrm{coker}(\hbar\otimes \Og_L)),$
$\ell_{\Og_L}(\ker (\lambda\otimes \Og_L)) = \ell_{\Og_L}(\mathrm{coker}(\lambda\otimes\Og_L)),$ and $\ell_{\Og_L}(\ker (\widetilde{\lambda}\otimes \Og_L)) = \ell_{\Og_L}(\mathrm{coker}(\widetilde{\lambda}\otimes\Og_L)).$ 
In particular, we have 
$$\ell_{\Og_L} (\mathrm{coker}\, \phi) = \ell_{\Og_L} (\ker(\hbar \otimes \Og_L)) + \ell_{\Og_L}(\mathrm{coker}\, \omega),$$
$$\ell_{\Og_L}(\mathrm{coker}\, \chi)= \ell_{\Og_L}(\ker(\lambda\otimes \Og_L)) + \ell_{\Og_L}(\mathrm{coker}\, \xi),$$ and
$$\ell_{\Og_L}(\mathrm{coker}\, \widetilde{\chi})= \ell_{\Og_L}(\ker(\widetilde{\lambda}\otimes \Og_L)) + \ell_{\Og_L}(\mathrm{coker}\,\widetilde{\xi}).$$ Plugging this into equation (\ref{zeroeq1}) gives
$$-\ell_{\Og_L}(\mathrm{coker}\, \omega) + \ell_{\Og_L}(\mathrm{coker}\, \xi) - \ell_{\Og_L}(\mathrm{coker} \, \widetilde{\xi})=0,$$ which implies the claim. \qed
\subsection{Exactness properties}
Once more we let $C$ be a proper reduced curve over $K$ whose Jacobian is semiabelian. For the sake of simplicity, we assume in this paragraph that $C$ is geometrically integral.  As above, we let $G:=\Pic^0_{C/K}$ and consider the exact sequence 
$0 \to T \to G \to E \to 0,$ where $E:=\Pic^0_{\widetilde{C}/K}$ is an Abelian variety and $T$ denotes the torus which comes from the singularities of $C.$ Let $\mathscr{T},$ $\mathscr{N},$ and $\widetilde{\mathscr{N}}$ be the Néron lft-models of $T,$ $G,$ and $E,$ respectively. As noted at the beginning of this article, the sequence
$$0 \to \mathscr{T} \to \mathscr{N} \to \widetilde{\mathscr{N}} \to 0$$ need not be exact. However, in our situation, we can at least measure the failure of exactness to some extent as follows: Let $\psi \colon \widetilde{\mathscr{C}} \to \mathscr{\mathscr{C}}$ be a strong model pair of $C$ over $\Og_K$ such that $\mathscr{C}$ is semi-factorial. Corollary \ref{Nérisomcor} tells us that there is a morphism
$$H^1(\mathscr{C}, \Og_{\mathscr{C}})\to \Lie\mathscr{N}$$ which induces the obvious identification at the generic fibre. We regard this morphism as an element of $D^b(\Og_K)$ by declaring that $H^1(\mathscr{C}, \Og_{\mathscr{C}})$ sit in degree 0. 
\begin{definition}
We shall denote the complex just constructed by $\mathbf{L}(\mathscr{C}).$ 
\end{definition}
This (perfect) complex turns out to contain much information about the exactness of the induced sequence of Néron models: 
\begin{proposition}\label{qiprop}
(i) The morphism $\mathscr{N} \to \widetilde{\mathscr{N}}$ is surjective in the fppf-topology; in particular, it is faithfully flat. \\
(ii) Suppose that $\psi\colon \widetilde{\mathscr{C}} \to \mathscr{C}$ is a strong model pair of $C$ with $\mathscr{C}$ semi-factorial. We obtain an induced map $\psi^\ast \colon \mathbf{L}(\mathscr{C}) \to \mathbf{L}(\widetilde{\mathscr{C}})$ such that the following are equivalent:\\
(a) The map $\mathscr{N} \to \widetilde{\mathscr{N}}$ is smooth,\\
(b) the sequence $0\to \mathscr{T} \to \mathscr{N} \to \widetilde{\mathscr{N}} \to 0$ is exact, \\
(c) the morphism $\psi^\ast$ is a quasi-isomorphism. \\
Moreover, we have $H^0(\mathbf{L}(\mathscr{C})) = H^1(\mathscr{C}, \Og_{\mathscr{C}})_{\mathrm{tors}}$ (and similarly for $\widetilde{\mathscr{C}}$). 
\end{proposition}
\begin{proof}
The surjectivity in part (i) follows from the fact that $P_{\mathscr{C}} \to P_{\widetilde{\mathscr{C}}}$ is surjective (Proposition \ref{Otherpaperprop}) together with the fact that $\mathscr{N}$ and $\widetilde{\mathscr{N}}$ are quotients of $P_{\mathscr{C}}$ and $P_{\widetilde{\mathscr{C}}},$ respectively (Corollary \ref{Nérisomcor}). This also implies that the map $\mathscr{N} \to \widetilde{\mathscr{N}}$ is fibre-wise flat, so it is faithfully flat by the fibre-wise criterion of flatness. \\
(ii) (a)$\Rightarrow$(b): If $\mathscr{N} \to \widetilde{\mathscr{N}}$ is smooth, it is elementary to show that its kernel is smooth and satisfies the Néron mapping property. \\
(b)$\Rightarrow$(c): If the sequence is exact, the map $\mathscr{N} \to \widetilde{\mathscr{N}}$ is smooth. Hence the map $\Lie \mathscr{N} \to \Lie \widetilde{\mathscr{N}}$ is surjective, which means that the map $H^1(\mathbf{L}(\mathscr{C}))\to H^1(\mathbf{L}(\widetilde{\mathscr{C}}))$ is surjective as well. By Theorem \ref{samelengththm}, the two potentially non-trivial cohomology modules of $\mathbf{L}(\mathscr{C})$ and $\mathbf{L}(\widetilde{\mathscr{C}})$ have the same lengths. Let $t$ (resp. $\widetilde{t}$) be this common length of the cohomology modules of $\mathbf{L}(\mathscr{C})$ (resp. $\mathbf{L}(\widetilde{\mathscr{C}})$). We have just shown that $t\geq \widetilde{t}.$ However, it is easy to see that the map $H^0(\mathbf{L}(\mathscr{C}))\to H^0(\mathbf{L}(\widetilde{\mathscr{C}}))$ is injective, so that $t\leq \widetilde{t}.$ Hence $t=\widetilde{t},$ and the injectivity and surjectivity we have just seen proves that $\psi^\ast$ is a quasi-isomorphism. \\
(c)$\Rightarrow$(a): Let $\mathscr{K}$ be the kernel of the map $\mathscr{N}\to \widetilde{\mathscr{N}}$ and consider the commutative diagram 
$$\begin{CD}
0@>>> \Lie \mathscr{T} @>>> H^1(\mathscr{C}, \Og_{\mathscr{C}}) @>>> H^1(\widetilde{\mathscr{C}}, \Og_{\widetilde{\mathscr{C}}}) @>>> 0\\
&&@VVV@VVV@VVV\\
0@>>> \Lie \mathscr{K}@>>> \Lie \mathscr{N} @>>> \Lie \widetilde{\mathscr{N}}.
\end{CD}$$
We obtain the exact sequence
$$H^1(\mathbf{L}(\mathscr{C})) \to H^1(\mathbf{L}(\widetilde{\mathscr{C}})) \to \mathrm{coker} (\Lie \mathscr{N} \to \Lie \widetilde{\mathscr{N}}) \to 0.$$
If $\psi^\ast$ is a quasi-isomorphism, the map $\Lie \mathscr{N} \to \Lie \widetilde{\mathscr{N}}$ must be surjective, so $ \mathscr{N} \to  \widetilde{\mathscr{N}}$ is smooth (this can be reduced to the finite type case, which follows from \cite[Proposition 1.1(e)]{LLR}). 
\end{proof}\\
This proposition can be used to explain a phenomenon already observed in \cite{OvI, Ov}: 
\begin{proposition}
With the same notation as in the previous proposition, if $\widetilde{\mathscr{C}}$ is cohomologically flat in dimension zero, then the sequence $0 \to \mathscr{T} \to \mathscr{N} \to \widetilde{\mathscr{N}} \to 0$ is exact. 
\end{proposition}
\begin{proof}
By the discussion following the proof of Proposition \ref{modpairexprop}, $\mathscr{C}$ is cohomologically flat in dimension zero as well. In particular, $H^0(\mathbf{L}(\mathscr{C}))=H^0(\mathbf{L}(\widetilde{\mathscr{C}}))=0.$ Hence, by Theorem \ref{samelengththm}, $H^1(\mathbf{L}(\mathscr{C}))=H^1(\mathbf{L}(\widetilde{\mathscr{C}}))=0.$ This shows $\mathbf{L}(\mathscr{C})\cong 0 \cong \mathbf{L}(\widetilde{\mathscr{C}}).$ Therefore the claim follows from the previous proposition.
\end{proof}\\
$\mathbf{Remark}.$ While it is clear that $\mathbf{L}(\mathscr{C})$ only depends on $C,$ there does not seem to be a general method to calculate the cohomology of this complex in terms of $C.$ It would be very interesting to have such a general procedure.
\subsubsection{An example} \label{examplepar}
We shall now construct an explicit example of a proper geometrically integral curve $X$ over $K$ such that the induced sequence of Néron models is not exact. This will answer the question posed in \cite[Remark after the proof of Theorem 3.10]{Ov}, negatively. In fact, such examples exist in any positive residue characteristic. Let $p>0$ be a prime number and assume that $\Og_K$ has residue characteristic $p.$ Let $\mathscr{E}$ be an elliptic curve over $\Og_K$ with ordinary generic fibre and supersingular special fibre. Let $\widetilde{X}$ be a torsor for $\mathscr{E}_K$ such that the associated class in $H^1(K, \mathscr{E}_K)$ has order $p.$ Such torsors exist by \cite[Theorem 6.6 and Corollary 6.7]{LLR}. Let $\widetilde{\mathscr{X}}\to \Spec \Og_K$ be the minimal proper regular model of $\widetilde{X}.$ Then we know the following:
\begin{itemize}
\item the special fibre of $\widetilde{\mathscr{X}}\to \Spec \Og_K$ has multiplicity $p$ \cite[Theorem 6.6]{LLR},
 \item the map $\widetilde{\mathscr{X}}\to \Spec \Og_K$ is not cohomologically flat in dimension zero \cite[Théorème 9.4.1(ii)]{RaynaudPic}
\item the group scheme $\mathscr{E}$ acts on $\widetilde{\mathscr{X}}$ (extending the torsor structure on the generic fibre), and the reduced special fibre of $\widetilde{\mathscr{X}} \to \Spec \Og_K$ is a (non-principal) homogeneous space for $\mathscr{E}_k;$ in particular, it is an elliptic curve \cite[p. 72]{RaynaudPic}.
\end{itemize}
Let $\mathscr{J}\subseteq \Og_{\widetilde{\mathscr{X}}}$ be the ideal which cuts out $(\widetilde{\mathscr{X}}_k)_{\mathrm{red}}.$ Let $f\in \Gamma(\widetilde{\mathscr{X}}_k, \Og_{\widetilde{\mathscr{X}}_k})$ be a non-constant function. We identify $\Og_{\widetilde{\mathscr{X}}_k}$ with $\Og_{\widetilde{\mathscr{X}}}/\mathscr{J}^p.$ After subtracting a constant function from $f,$ we may assume that $f\in \mathscr{J}/\mathscr{J}^p.$ Choose a regular function $\widetilde{f}$ on some $\Og_K$-dense open subset $U\subseteq \widetilde{\mathscr{X}}$ which restricts to $f.$  Let $i\in \{1,..., p-1\}$ be maximal such that $\widetilde{f}\in \mathscr{J}^  i.$ By shrinking $U$ if necessary, we may assume that $\widetilde{f}$ generates $\mathscr{J}^  i \mid_U.$ 
Because $\widetilde{X}$ is smooth over $K,$ there exists a closed point $x$ on $\widetilde{X}$ such that $L:=\kappa(x)$ is separable over $K.$ Write $[L:K]=pn$ for some $n\in \N.$ Let $c\colon \Spec \Og_L\to \widetilde{\mathscr{X}}$ be the induced map; its scheme-theoretic image $\overline{x}$ is the Zariski closure of $x$ in $\widetilde{\mathscr{X}}.$ Note that, in particular, we have $\langle \overline{x}, (\widetilde{\mathscr{X}}_k)_{\mathrm{red}}\rangle =n,$ where $\langle -,- \rangle$ denotes the intersection product of divisors on $\widetilde{\mathscr{X}}.$ 
Now let $y$ be the closed point on the special fibre of $\widetilde{\mathscr{X}} \to \Spec \Og_K$ which also lies on $\overline{x}.$ Because $\mathscr{E}_k$ acts transitively on $(\widetilde{\mathscr{X}}_k)_{\mathrm{red}},$ there is a section $\sigma\in \mathscr{E}(\Og_K)$ such that $\sigma_k\cdot y \in U_k.$ In particular, we may assume that the map $c\colon \Spec \Og_L \to \widetilde{\mathscr{X}}$ factors through $U.$ 
Hence we know that $c^\ast\widetilde{f}=\epsilon \pi_L^ {in}$ for some $\epsilon \in \Og_L^\times.$ But because $in<[L:K],$ this does not restrict to a constant function in $\Og_L\otimes_{\Og_K}k.$ Now we replace $\widetilde{\mathscr{X}}$ by the scheme obtained by blowing up $\widetilde{\mathscr{X}}$ repeatedly in closed points on the special fibre until the map $c\colon \Spec \Og_L \to \widetilde{\mathscr{X}}$ becomes a closed immersion and its image $D$ intersects $(\widetilde{\mathscr{X}}_k)_{\mathrm{red}}$ transversely. Then we know that the map 
$$\psi\colon \widetilde{\mathscr{X}} \to \mathscr{X}:= \widetilde{\mathscr{X}} \cup_{\Spec\Og_L} \Spec \Og_K$$
is a strong model pair of the curve $X:=\widetilde{X}\cup_{\Spec L} \Spec K$ and that $\mathscr{X}$ is semi-factorial\footnote{The map is a strong model pair by construction. That $\mathscr{X}$ is semi-factorial follows as in  \cite[Theorem 4.3]{Ov}. Indeed, in order to guarantee that $\mathscr{X}$ is semi-factorial, we all we must ensure is that each irreducible component of $(\widetilde{\mathscr{X}}_k)_{\mathrm{red}}$ intersects at most one connected component of $D.$ But this is trivial in our situation since $D$ is connected.}. Taking the cohomology of the exact sequence $0 \to \Og_{\mathscr{X}} \overset{\cdot \pi_K}\to\Og_{\mathscr{X}} \to \Og_{\mathscr{X}_k} \to 0$ shows that $$\Gamma(\mathscr{X}_k, \Og_{\mathscr{X}_k})/k =H^1(\mathscr{X}, \Og_{\mathscr{X}})[\pi_K];$$ the analogous statement holds for $\widetilde{\mathscr{X}}$ for the same reason. Since the push-out used to construct $\mathscr{X}$ commutes with arbitrary base change, we see that the map $\Gamma(\mathscr{X}_k, \Og_{\mathscr{X}_k}) \to \Gamma(\widetilde{\mathscr{X}}_k, \Og_{\widetilde{\mathscr{X}}_k})$ is not an isomorphism, which implies that $H^1(\mathscr{X}_k, \Og_{\mathscr{X}_k})_{\mathrm{tors}} \to H^1(\widetilde{\mathscr{X}}_k, \Og_{\widetilde{\mathscr{X}}_k})_{\mathrm{tors}} $ is not an isomorphism either. In particular, the map $\mathbf{L}(\mathscr{X}) \to \mathbf{L}(\widetilde{\mathscr{X}})$ is not a quasi-isomorphism and the sequence of Néron models induced by the exact sequence $$0\to \Res_{L/K}\Gm/\Gm \to \Pic^0_{X/K}\to \Pic^0_{\widetilde{X}/K} \to 0$$ is not exact by Proposition \ref{qiprop}.
\section{Existence of Néron (lft-)models}
\subsection{The structure of Jacobians in degree of imperfection 1}
Let $\kappa$ be an arbitrary field of characteristic $p>0.$ Throughout this section, $C$ will denote a proper curve over $\kappa;$ if $C$ is reduced, we let $\widetilde{C}$ be its normalisation. Let $G:=\Pic^0_{C/K},$ let $\mathscr{R}_{us, \kappa}(G)$ denote the maximal smooth connected split unipotent closed subgroup of $G$ (cf. \cite[p. 63]{CGP}), and let $\mathrm{uni}(G)$ denote the maximal unirational subgroup of $G$ over $K$ (cf. \cite[p. 310]{BLR}). We shall begin by studying the structure of $\Pic^0_{C/\kappa}.$ The results and their proofs are similar to the corresponding ones in \cite{Ov}, so we shall only indicate the necessary changes.
\begin{proposition}
Suppose that $\delta(\kappa) \leq 1$ and that $C$ is reduced. Then $\mathrm{uni}(\Pic^0_{\widetilde{C}/\kappa})=0.$
\end{proposition}
\begin{proof}
We have $$\Pic^0_{\widetilde{C}/\kappa}=\Res_{\Gamma(\widetilde{C}, \Og_{\widetilde{C}})/\kappa}\Pic^0_{\widetilde{C}/\Gamma(\widetilde{C}, \Og_{\widetilde{C}})}.$$ Because $\widetilde{C}$ is normal and $\delta(\kappa)\leq1,$ $\widetilde{C}$ is geometrically reduced over $\Gamma(\widetilde{C}, \Og_{\widetilde{C}})$ (Proposition \ref{Steinredprop}). Because the base change of a unirational algebraic group is unirational, the claim now follows from \cite[Proposition 2.32]{Ov}. 
\end{proof}
\begin{theorem}
Suppose $C$ is reduced and let $\widetilde{C} \overset{\widetilde{\varsigma}}{\to} C^{\mathrm{sn}} \overset{{\varsigma}}{\to} C$ be the canonical factorisation\footnote{See \cite[Lemma 2.14 and Theorem 2.24(i)]{Ov}.} of the normalisation map $\nu,$ where $C^{\mathrm{sn}}$ is the seminormalisation of $C.$ Then we have a filtration
$$0 \subseteq \ker {\varsigma}^\ast \subseteq \ker \nu^\ast \subseteq \Pic^0_{C/\kappa},$$ 
where $\ker {\varsigma}^\ast$ is split unipotent and $\ker \nu^\ast$ is unirational. If $\delta(\kappa) \leq 1,$ $\ker {\varsigma}^\ast = \mathscr{R}_{us, \kappa}(\Pic^0_{C/\kappa})$ and $\ker \nu^\ast = \mathrm{uni}(\Pic^0_{C/\kappa}).$ \label{Structurethm}
\end{theorem}
\begin{proof}
That $\ker \nu^\ast$ is unirational (resp. equals $\mathrm{uni}(\Pic^0_{C/\kappa})$ if $\delta(\kappa)\leq1$) can be shown as in the proof of \cite[Theorem 2.33]{Ov}. For the first claim (resp. equality), we may assume without loss of generality that $\kappa$ is separably closed (\cite[Corollary B.3.5]{CGP} together with \cite[Lemma 2.25]{Ov} and Lemma \ref{deltafinsepeqlem} above). The first claim follows as in the proof of \cite[Theorem 2.33]{Ov}. If $\delta(\kappa)\leq1,$ we can again argue as in \it loc. cit. \rm as soon as we can show that $\ker \widetilde{\varsigma}^\ast$ contains no closed subgroup isomorphic to $\Ga.$ By \cite[proof of Theorem 2.24(ii)]{Ov} and \cite[Proposition 2.30]{Ov}, we see that $\ker \widetilde{\varsigma}^\ast$ is a repeated extension of algebraic $\kappa$-groups of the form $\Res_{A/\kappa}\Gm/(\Res_{B/\kappa}\Gm \cdot \Res_{L/\kappa} \Gm),$ where $A$ is either a field or the product of a finite extension of $\kappa$ with itself, and $B$ and $L$ are subalgebras of $A.$ In the first case, both $B$ and $L$ are fields. Since $\delta(\kappa)\leq1$ and $\kappa$ is separably closed, each finite extension of $\kappa$ is isomorphic to $\kappa^{1/p^n}$ for some $n.$ Hence either $L\subseteq B$ or $B\subseteq L$ in a canonical way. Denoting the larger field by $M,$ we see that the quotient above is, in fact, isomorphic to $\Res_{A/\kappa}\Gm/\Res_{M/\kappa}\Gm.$ In the second case, we write $A=F\times F $ for a field extension $F\supseteq \kappa.$ The argument from \it loc. cit. \rm shows moreover that we can take $B=F$ with the map $B\to A$ being the diagonal map. Then $L$ is either a field of the product of two field extensions $L_1$ and $L_2$ of $\kappa.$ In the first case, the quotient above is isomorphic to $\Res_{F/\kappa}\Gm;$ in the second case it is isomorphic to $\Res_{F/\kappa}\Gm/\Res_{M/\kappa}\Gm,$ where $M$ is the larger one of $L_1$ and $L_2.$ Clearly, we have $\Hom_{\kappa}(\Ga, \Res_{F/\kappa} \Gm)=0.$ Hence all we need to show is that, if $M\subseteq A$ are finite reduced $\kappa$-algebras, we have $$\Hom_{\kappa}(\Ga, \Res_{A/\kappa}\Gm/\Res_{M/\kappa}\Gm)=0.$$ Any element $\tau$ of this group gives an element $\sigma\in\Ext^1_{\kappa}(\Ga, \Res_{M/\kappa}\Gm).$ However, since $M$ is finite over $\kappa,$ the Grothendieck spectral sequence shows that 
$$\Ext^1_{\kappa}(\Ga, \Res_{M/\kappa}\Gm)=\Ext^1_{M}(\Ga, \Gm)=0$$ \cite[Exposé XVII, Théorème 6.1.1 A) ii)]{SGA3}. Hence $\sigma=0$ and $\tau$ lifts to a map $\Ga \to \Gm$ over $A,$ which must vanish since $A$ is a finite product of fields. 
\end{proof}
\begin{corollary}
Suppose $C$ is reduced. If $\mathrm{uni}(\Pic^0_{C/\kappa})=0,$ then the map $\Pic^0_{C/\kappa} \to \Pic^0_{\widetilde{C}/\kappa}$ is an isomorphism. If $\mathscr{R}_{us, \kappa}(\Pic^0_{C/\kappa})=0,$ then $C$ is seminormal. If $\delta(\kappa)\leq 1,$ the converse of both statements holds as well. \label{structurecor}
\end{corollary}
\begin{proof}
The first claim follows immediately from Theorem \ref{Structurethm}. Moreover, the proof of Proposition \ref{seminormalpropII} shows that, if $C$ is not seminormal, then $\ker \varsigma^\ast $ is non-trivial. Hence the second claim also follows.
\end{proof}
\subsection{Existence results}
Now let $S$ be an excellent Dedekind scheme of equal characteristic $p>0$ with field of fractions $K$ and \it perfect \rm residue fields. By Corollary \ref{Deltacor}, we have $\delta(K)=1.$ As in \cite{Ov}, we now have
\begin{theorem} \rm (cf. \cite[Chapter 10.3, Conjecture II]{BLR}) \it 
Let $C$ be a proper curve over $K$ such that $\mathrm{uni}(\Pic^0_{C/K})=0.$ Then $\Pic^0_{C/K}$ admits a Néron model over $S.$ 
\end{theorem}
\begin{proof}
We may assume without loss of generality that $C$ is reduced. Indeed, our assumption implies that the map $\Pic^0_{C/K} \to \Pic^0_{C_{\mathrm{red}}/K}$ is an isomorphism (\cite[Chapter 9.2, Proposition 5]{BLR} together with \cite[Lemma 2.3]{Ov}). Then we know that the map $\Pic^0_{C/K} \to \Pic^0_{\widetilde{C}/K}$ is an isomorphism (Corollary \ref{structurecor}). By Proposition \ref{Steinredprop}, we know that $\widetilde{C}$ is geometrically reduced over $\Gamma(\widetilde{C}, \Og_{\widetilde{C}}).$ If $S'$ denotes the integral closure of $S$ in $\Gamma(\widetilde{C}, \Og_{\widetilde{C}}),$ then \cite[Theorem 3.4]{Ov} shows that $\Pic^0_{\widetilde{C}/\Gamma(\widetilde{C}, \Og_{\widetilde{C}})}$ admits a Néron model $\widetilde{\mathscr{N}'}\to S'.$ Therefore the scheme
$$\widetilde{\mathscr{N}}:=\Res_{S'/S} \widetilde{\mathscr{N}}'$$
is a Néron model of $\Res_{\Gamma(\widetilde{C}, \Og_{\widetilde{C}})/K}\Pic^0_{\widetilde{C}/\Gamma(\widetilde{C}, \Og_{\widetilde{C}})}=\Pic^0_{\widetilde{C}/K}=\Pic^0_{C/K}$ over $S.$
\end{proof}\\
The following lemma is surely well-known (cf. \cite[Chapter 8.3, Lemma 3.49]{Liu} or \cite[Tag 0BG6]{Stacks} for similar statements). We include a proof for the reader's convenience.
\begin{lemma}
Let $R$ be an excellent discrete valuation ring and let $A$ be an integral regular flat $R$-algebra of finite type. Let $\widehat{R}$ be the completion of $R.$ Then $A\otimes_R\widehat{R}$ is regular. \label{stillregularlem}
\end{lemma}
\begin{proof}
Let $\mathfrak{p}\subseteq A\otimes_R\widehat{R}$ be a prime ideal. We must show that the localisation $(A\otimes_R\widehat{R})_{\mathfrak{p}}$ is regular. If $\mathfrak{p}$ maps to the generic point of $\Spec \widehat{R},$ this is clear because the extension $\Frac R \subseteq \Frac \widehat{R}$ is separable. Otherwise we put $\mathfrak{q}:=\mathfrak{p}\cap A$ and observe that the canonical map $A/\mathfrak{q} \to (A\otimes_R\widehat{R})/\mathfrak{p}$ is an isomorphism since $R\to \widehat{R}$ induces an isomorphism of residue fields. This shows that $\mathfrak{p}=\mathfrak{q}\otimes_R\widehat{R}$ and hence $\mathfrak{p}/\mathfrak{p}^2=\mathfrak{q}/\mathfrak{q}^2;$ moreover we see that the map $\kappa(\mathfrak{q})\to \kappa(\mathfrak{p})$ on residue fields is an isomorphism. Now choose a maximal chain $\mathfrak{p}_0\subset ... \subset \mathfrak{p}_n=\mathfrak{p}$ of prime ideals such that $n= \mathrm{ht}_{A\otimes_R\widehat{R}}\, \mathfrak{p}.$ Clearly $\mathfrak{p}_0$ is minimal, so $(A\otimes_R\widehat{R})/\mathfrak{p}_0$ is generically of the same dimension as $A\otimes_R\Frac{R}.$ Moreover, we see that
$$\mathrm{ht}_{(A\otimes_R\widehat{R})/\mathfrak{p}_0} \, \mathfrak{p}/\mathfrak{p}_0 = \mathrm{ht}_{A\otimes_R\widehat{R}} \, \mathfrak{p}.$$
But $R$ is universally catenary, so the dimension formula \cite[Tag 02IJ]{Stacks} shows that 
$\mathrm{ht}_A\, \mathfrak{q} = \mathrm{ht}_{(A\otimes_R\widehat{R})/\mathfrak{p}_0}\,\mathfrak{p}/\mathfrak{p}_0.$ From this, we deduce (replacing $\mathfrak{p}$ and $\mathfrak{q}$ with their respective localisations) that
$$\dim (A\otimes_R\widehat{R})_{\mathfrak{p}} = \dim A_{\mathfrak{q}} = \dim_{\kappa(\mathfrak{q})} \mathfrak{q}/\mathfrak{q}^2 = \dim_{\kappa(\mathfrak{p})} \mathfrak{p}/\mathfrak{p}^2.$$
\end{proof}\\
Finally, we have
\begin{theorem} \rm (cf. \cite[Chapter 10.3, Conjecture I]{BLR}) \it
Let $C$ be a proper curve over $K$ such that $\mathscr{R}_{us, K}(\Pic^0_{C/K})=0.$ Then $\Pic^0_{C/K}$ admits a Néron lft-model over $S.$ 
\end{theorem}
\begin{proof}
We may once again assume that $C$ is reduced. Our assumption implies that $\Pic^0_{C/K}$ admits Néron lft-models $\mathscr{N}_{\mathfrak{p}}$ at all closed points $\mathfrak{p}$ of $S$ \cite[Chapter 10.2, Theorem 2(b')]{BLR}.  By \cite[Chapter 10.1, Proposition 9]{BLR}, we may replace $S$ by a dense open subscheme (in particular, we may suppose that $S$ is affine). Moreover, \it loc. cit. \rm also shows that it suffices to find a finite locally free $\Gamma(S, \Og_S)$-module $\mathscr{L}$ and an isomorphism $\mathscr{L}\otimes_{\Gamma(S, \Og_S)} K \to \Lie \Pic^0_{C/K}=H^1(C, \Og_C)$ such that, for all closed points $\mathfrak{p}\in S,$ the images of $\Lie \mathscr{N}_{\mathfrak{p}}$ and $\mathscr{L}\otimes_{\Gamma(S, \Og_S)} \Og_{S, \mathfrak{p}}$ inside $\mathscr{L}\otimes_{\Gamma(S, \Og_S)} K$ coincide. Using the method from \cite[Corollary 3.7]{Ov}, we construct proper and flat\footnote{(though \it a priori \rm not necessarily cohomologically flat)} models $\mathscr{C}\to S$ and $\widetilde{\mathscr{C}}\to S$ of $C$ and $\widetilde{C},$ respectively, which fit into a co-cartesian diagram
$$\begin{CD}
\widetilde{\mathscr{C}}@>>>\mathscr{C}\\
@AAA@AAA\\
\Spec \mathcal{A}@>>>\Spec\mathcal{B},
\end{CD}$$
where $\mathcal{B}\subseteq \mathcal{A}$ are finite regular $\Gamma(S, \Og_S)$-algebras, such that the vertical maps are closed immersions. Replacing $S$ by a dense open subscheme if necessary, we may suppose that $H^1(\mathscr{C}, \Og_{\mathscr{C}})$ is locally free over $\Gamma(S, \Og_S);$ we shall now show that 
$$\mathscr{L}:=H^1(\mathscr{C}, \Og_{\mathscr{C}})$$ has the required properties. By Lemma \ref{stillregularlem}, we see that the scheme\footnote{For a closed point $\mathfrak{p}\in S$ and a morphism $X\to S,$ we put $X_{\mathfrak{p}}:=X\times_S\Spec \widehat{\Og^{\mathrm{sh}}_{S,{\mathfrak{p}}}}.$} $\widetilde{\mathscr{C}}_{\mathfrak{p}}$ is a proper, flat, and regular model of $\widetilde{C}\times_K\Spec \widehat{K^{\mathrm{sh}}}$ over $\Spec \widehat{\Og^{\mathrm{sh}}_{S, \mathfrak{p}}}.$ Using the method from the proofs of \cite[Theorem 4.3 and Proposition 4.4]{Ov}, we construct a proper birational morphism $\widetilde{\mathscr{C}}'_{\mathfrak{p}}\to \widetilde{\mathscr{C}}_{\mathfrak{p}}$ such that the push-out
$$\mathscr{C}'_{\mathfrak{p}}:=\widetilde{\mathscr{C}}'_{\mathfrak{p}}\cup_{\Spec \mathcal{A}_{\mathfrak{p}}} \Spec \mathcal{B}_{\mathfrak{p}}$$ is a proper, flat, and semi-factorial model of $C\times_K\Spec \widehat{K^{\mathrm{sh}}}.$ Note that we have a canonical map $\mathscr{C}'_{\mathfrak{p}}\to \mathscr{C}_{\mathfrak{p}}.$ The Lemma of five homomorphisms applied to the exact sequence (\ref{cohomseq}) from subsection \ref{Chaiproofsubsec} shows that the map $H^1(\mathscr{C}_{\mathfrak{p}}, \Og_{\mathscr{C}_{\mathfrak{p}}}) \to H^1(\mathscr{C}'_{\mathfrak{p}}, \Og_{\mathscr{C}'_{\mathfrak{p}}})$ is an isomorphism. Because the residue fields of $S$ are perfect and $H^1(\mathscr{C}, \Og_{\mathscr{C}})$ is torsion-free, Corollary \ref{lengthscor} (together with \cite[Chapter 10.1, Proposition 3]{BLR}) shows that the map $$H^1(\mathscr{C}'_{\mathfrak{p}}, \Og_{\mathscr{C}'_{\mathfrak{p}}}) \to (\Lie \mathscr{N}_{\mathfrak{p}})\otimes_{\Og_{S, \mathfrak{p}}} \widehat{\Og^{\mathrm{sh}}_{S,{\mathfrak{p}}}}$$ is an isomorphism. Now observe that $\mathscr{C}\times_S\Spec \Og_{S, \mathfrak{p}}$ is cohomologically flat in dimension zero over $\Og_{S, \mathfrak{p}}$ for all closed points $\mathfrak{p}\in S$ \cite[Chapter 5.3, Corollary 3.22 and Remark 3.30]{Liu}, so $\Pic_{\mathscr{C}\times_S\Spec \Og_{S, \mathfrak{p}}/\Og_{S, \mathfrak{p}}}$ is a smooth algebraic space over $\Og_{S, \mathfrak{p}}$ (\cite[Chapter 8.3, Theorem 1]{BLR} and \cite[Corollaire 2.3.2]{RaynaudPic}). It is easy to see that the Néron mapping property remains valid for algebraic spaces, so there is a canonical map $\mathscr{L}\otimes_{\Gamma(S, \Og_S)} \Og_{S, \mathfrak{p}} \to \Lie \mathscr{N}_{\mathfrak{p}}.$ Our claim follows once we establish that this map is an isomorphism; however, this can be checked after applying the functor $-\otimes_{\Og_{S, \mathfrak{p}}} \widehat{\Og^{\mathrm{sh}}_{S,{\mathfrak{p}}}}.$ 
\end{proof}\\
We are now able to state the following general result on Néron lft-models of Jacobians, generalising Raynaud's theory \cite{RaynaudPic}: 
\begin{theorem}
Let $R$ be a discrete valuation ring with field of fractions $K$ and residue field $\kappa$ (which we do not assume to be perfect). Let $\widehat{R}$ be the completion of $R$ and denote by $\widehat{K}$ its field of fractions. Let $C$ be a proper reduced curve over $K$ and assume that the normalisation $\widetilde{C}$ admits a proper regular model $\widetilde{\mathscr{C}} \to \Spec R.$ Let $\Spec B$ be the singular locus of $C$ endowed with its reduced subscheme structure and let $\Spec A\subseteq \widetilde{C}$ be its scheme-theoretic preimage. Finally, assume that one of the conditions is satisfied: 
\begin{itemize}
\item each irreducible component of $\widetilde{C}_{K^{\mathrm{sh}}}$ has a $K^{\mathrm{sh}}$-point, or
\item $\kappa$ is perfect. 
\end{itemize}
Then the following are equivalent: \\
(i) $\Pic^0_{C/K}$ admits a Néron lft-model over $R,$ \\
(ii) $C_{\widehat{K}}$ is seminormal, and \\
(iii) $C$ is seminormal and $A\otimes_K\widehat{K}$ is reduced. \\
Moreover, if these conditions are satisfied, then after modifying $\widetilde{\mathscr{C}}$ if necessary, there exists a finite morphism $\psi\colon \widetilde{\mathscr{C}} \to \mathscr{C}$ over $R$ such that $\psi_{R^{\mathrm{sh}}}\colon \widetilde{\mathscr{C}}_{R^{\mathrm{sh}}} \to \mathscr{C}_{R^{\mathrm{sh}}}$ is a strong model pair of $C_{K^{\mathrm{sh}}}$ and such that $\mathscr{C}_{R^{\mathrm{sh}}}$ is semi-factorial. In particular, $P_{\mathscr{C}/R}^{\mathrm{sep}}$ is the Néron lft-model of $\Pic^0_{C/K}.$ \label{Raynaudtypethm}
\end{theorem}
\begin{proof}
Throughout the proof, we may assume without loss of generality that $C$ is connected. We begin by observing that $\widetilde{C}_{\widehat{K}}$ is regular. Indeed, the argument from the proof of Lemma \ref{stillregularlem} shows that $\widetilde{\mathscr{C}}_{\widehat{R}}$ is regular at all points on the special fibre. However, since $\widetilde{\mathscr{C}}_{\widehat{R}}$ is excellent, its singular locus is closed. Because this scheme is proper over $\widehat{R},$ the singular locus would have to intersect the special fibre if it were non-empty. Hence $\widetilde{\mathscr{C}}_{\widehat{R}}$ is already regular (see also \cite[Tag 0BG6 (3)]{Stacks}). Since the morphism $\psi_K\colon \widetilde{C} \to C$ is scheme-theoretically dominant, so is $\psi_K\times_K \mathrm{Id}_{\widehat{K}}.$ In particular, $C_{\widehat{K}}$ is reduced.\\
(i) $\iff$ (ii): By \cite[Chapter 10.2, Theorem 2]{BLR}, (i) is equivalent to $\mathscr{R}_{us, \widehat{K}}(\Pic^0_{C_{\widehat{K}}/\widehat{K}})=0.$ Because we already know that $C_{\widehat{K}}$ is reduced, this is equivalent to (ii) by Corollary \ref{structurecor} (if $\kappa$ is perfect; see Lemma \ref{plusonelem}) and \cite[Corollary 2.34]{Ov} (note that the existence of $K^{\mathrm{sh}}$-points as in the Theorem implies geometric reducedness). \\
(iii) $\Rightarrow$ (ii): If (iii) holds, then $C$ arises as the push-out of $\Spec A \to \widetilde{C}$ along $\Spec A \to \Spec B$ \cite[Theorem 2.24(ii)]{Ov}. Hence $C_{\widehat{K}}$ arises as the push-out of $\Spec A\otimes_K\widehat{K}\to \widetilde{C}_{\widehat{K}}$ along $\Spec A\otimes_K\widehat{K} \to \Spec B\otimes_K\widehat{K}.$ Therefore our assumptions imply that $C_{\widehat{K}}$ is seminormal. \\
(i) $\Rightarrow$ (iii): If $\Pic^0_{C/K}$ admits a Néron lft-model, then $\mathscr{R}_{us,K}(\Pic^0_{C/K})=0.$ This follows from \cite[Chapter 10.2, Theorem 2]{BLR} together with the fact that $\mathscr{R}_{us, K}(\Pic^0_{C/K})_{\widehat{K}}\subseteq \mathscr{R}_{us, \widehat{K}}(\Pic^0_{C_{\widehat{K}}/\widehat{K}}).$ However, $\mathscr{R}_{us, K}(\Pic^0_{C/K})$ vanishes only if $C$ is seminormal (Corollary \ref{structurecor}). In particular, $C$ arises as the push-out of $\Spec A \to \widetilde{C}$ along $\Spec A \to \Spec B$ \cite[Theorem 2.24]{Ov}. Hence $C_{\widehat{K}}$ is the push-out of $\Spec A\otimes_K\widehat{K} \to \widetilde{C}_{\widehat{K}}$ along $\Spec A\otimes _K\widehat{K} \to \Spec B\otimes_K\widehat{K},$ so if $A\otimes_K\widehat{K}$ is non-reduced, then $C_{\widehat{K}}$ is not seminormal. Indeed, Lemma \ref{nilplem} below shows that we can find a nilpotent element $y\in A\otimes_K\widehat{K}$ not contained in $B\otimes_K\widehat{K}.$ In particular, there exists $N_0\in \N$ such that, for all $n> N_0,$ we have $y^n\in B\otimes_K\widehat{K}.$ Choose $N_0$ minimal with this property. Then $N_0\geq 1$ by assumption. Now replace $y$ by $y^{N_0}.$ Then $y\not\in B\otimes_K\widehat{K},$ but $y^2$ and $y^3$ both lie in $B\otimes_K\widehat{K}.$ If $U$ is an open affine neighbourhood of $\Spec A\otimes_K\widehat{K}$ in $\widetilde{C}_{\widehat{K}},$ then we can find $f\in \Gamma(U, \Og_U)$ whose image in $A$ is $y.$ The scheme $U':=U\cup_{\Spec A\otimes_K\widehat{K}} (\Spec B\otimes_K\widehat{K})$ is an open subscheme of $C_{\widehat{K}}$ by construction. Moreover, $f^2$ and $f^3$ are global functions on $U'$ but $f$ is not. Hence $U'$ (and therefore $C_{\widehat{K}}$) is not seminormal. This contradicts (i) $\iff$ (ii), which we have already proven. \\
We construct the morphism $\psi\colon \widetilde{\mathscr{C}} \to \mathscr{C}$ as in \cite[Theorem 4.3]{Ov}. Moreover, we replace the reference to the embedded resolution theorem as formulated in \cite[Chapter 9.2, Theorem 2.26]{Liu}, by a reference to \cite[Tag 0BIC]{Stacks}, in order to avoid the hypothesis on excellence imposed in \cite{Liu}. Note that the conditions of \cite[Tag 0BIC]{Stacks} are satisfied since $A\otimes_K\widehat{K}$ (and hence $B\otimes_K\widehat{K}$) are reduced, so the integral closures of $R$ in $A$ and $B$ are finite over $R.$ Our hypotheses further imply that the map $\Pic C_{K^{\mathrm{sh}}} \to \Pic_{C_{K^{\mathrm{sh}}}/K^{\mathrm{sh}}}(K^{\mathrm{sh}})$ is bijective. Since we already know that $\Pic^0_{C/K}$ admits a Néron lft-model $\mathscr{N},$ the proofs of Proposition \ref{fingennérprop} and Corollary \ref{Nérisomcor} can be taken \it mutatis mutandis \rm to show that $P_{\mathscr{C}_{R^{\mathrm{sh}}}/R^{\mathrm{sh}}} ^ \mathrm{sep} \to \mathscr{N}_{R^{\mathrm{sh}}}$ is an isomorphism. Indeed, if $G$ is a smooth connected algebraic group over a \it separably \rm closed field $\kappa,$ it is still true that $G(\kappa)$ is $n$-divisible for all positive integers $n$ invertible in $\kappa,$ so if $G(\kappa)$ is finitely generated then $G(\kappa)$ is finite. Hence $G=0$ since $G$ is smooth.
\end{proof}
\begin{lemma}
With the notation from Theorem \ref{Raynaudtypethm}, let $B\subseteq A$ be finite $K$-algebras such that the inclusion is not an equality and such that $B$ is a field. Moreover, assume that $A\otimes_K\widehat{K}$ is non-reduced. Then there exists a nilpotent element $y\in A\otimes_K\widehat{K}$ not contained in $B\otimes_K\widehat{K}.$ \label{nilplem}
\end{lemma}
\begin{proof}
Let $N_A$ and $N_B$ be the ideals of nilpotent elements in $A\otimes_K\widehat{K}$ and $B\otimes_K\widehat{K},$ respectively. Assume, for the sake of a contradiction, that the inclusion $N_B\subseteq N_A$ is an equality. Then $N_B\not=0$ by assumption. Now choose a $B$-basis $1, \alpha_1,..., \alpha_d$ of $A$ for some $d\in \N.$ Let $\nu\in N_B$ be any non-zero element. Since $1\otimes 1, \alpha_1\otimes 1,..., \alpha_d \otimes 1$ is a $B\otimes_K\widehat{K}$-basis of $A\otimes_K\widehat{K},$ the element $\nu \cdot \alpha_1\otimes 1 \in A\otimes_K\widehat{K}$ is not contained in $B\otimes_K\widehat{K}.$ However, it is clearly nilpotent, contradicting the assumption that $N_B=N_A.$ 
\end{proof}\\
In order to determine when $\Pic^0_{C/K}$ admits a Néron model (i. e. when the Néron lft-model is of finite type over $R$), we need to understand the dimension of the maximal split torus inside $\Pic^0_{C/K}$ \cite[Chapter 10.2, Theorem 2]{BLR}. This invariant turns out to be purely combinatorial in nature. We begin with the following 
\begin{definition}
Let $\kappa$ be an arbitrary field and let $C$ be a curve of finite type over $\kappa.$ Let $X_1,..., X_n$ be the irreducible components of $C;$ we fix the ordering on this set once and for all. \\
(i) A closed point $x$ on $C$ is a \rm multibranch point \it if $C$ is not unibranch at $x,$ and an \rm intersection point \it if $x$ is contained in at least two irreducible components of $C.$ \\
(ii) We define the \rm toric graph\footnote{We allow loops and multiple edges, hence some authors would call $\Gamma(C)$ a \it multigraph \rm or a \it pseudograph. \rm} \it $\Gamma=\Gamma(C)$ of $C$ as follows:
\begin{itemize}
\item The vertices of $\Gamma$ are the irreducible components $X_1,..., X_n$ of $C.$
\item For all intersection points $x\in C,$ let $\Sigma(x)$ be the set of all irreducible components containing $x.$ For $i\not=j,$ an edge connecting $X_i$ and $X_j$ is a pair $(\{X_i, X_j\}, x),$ such that $\{X_i, X_j\}\subseteq \Sigma(x)$ and such that there is no element between $X_i$ and $X_j$ in $\Sigma(x)$ with respect to the induced ordering.
\item For each $i=1,...,n,$ the number of loops centred at $X_i$ is equal to
$$\sum_{\text{$x\in X_i $ \rm  multibranch in $X_i$}} (n_x-1),$$
where $n_x$ is the number of branches of $X_i$ at $x$ \cite[Tag 0C38]{Stacks}.
\end{itemize}
\end{definition}
If $C$ and $C'$ are curves of finite type over $\kappa,$ then $\Gamma(C\sqcup C')\cong \Gamma(C)\sqcup \Gamma(C'),$ and one shows easily by induction on the number of irreducible components that if $C$ is connected, then so is $\Gamma(C).$
\begin{lemma}  \label{homeographlem} Let $C \to C'$ be a finite morphism of proper curves over $\kappa$ which is an homeomorphism (with compatible orderings on the sets of irreducible components of $C$ and $C'$) such that the induced map $C_{\mathrm{red}}^{\sim} \to C_{\mathrm{red}}^{' \sim}$ is an isomorphism. Then $\Gamma(C)$ and $\Gamma(C')$ are isomorphic.
\end{lemma}
\begin{proof}
This follows immediately from the definition together with \cite[Tag 0C1S (1)]{Stacks}. 
\end{proof}\\
$\mathbf{Remark}.$ Because the definition of $\Gamma$ involves choices, $\Gamma(C)$ is not functorial in $C.$ In particular, there is no canonical isomorphism $\Gamma(C) \to \Gamma(C')$ in the situation of the Lemma. 
\begin{lemma}  \label{alternatingsplitlem}
Let $0\to G_1 \to G_2 \to ... \to G_d \to 0$ be an exact sequence of smooth connected commutative algebraic groups over the field $\kappa.$ For $i=1,..., d,$ let $s_i$ be the dimension of the maximal \rm split \it torus inside $G_i.$ Then we have
$$\sum_{i=1} ^d (-1) ^ i s_i=0.$$
\end{lemma}
\begin{proof}
First observe that the cases $d=1$ and $d=2$ are trivial. By splitting the sequence into $0\to G_1\to G_2\to G_2/G_1\to 0$ and $0\to G_2/G_1 \to G_3 \to ... \to G_d \to 0,$ as well as using induction on $d,$ we reduce the claim to the case $d=3,$ which is proven in \cite[Lemma 3.2.2.4]{HN}.
\end{proof}\\
We are now ready to prove the following
\begin{theorem}
Let $\kappa$ be an arbitrary field and let $C$ be a proper curve over $\kappa.$ Let $s$ be the dimension of the maximal \rm split \it torus inside $\Pic^0_{C/\kappa}.$ Then 
$$s=\dim_{\Q} H_1(\Gamma(C), \Q) = E-V+N,$$
where $E$ is the number of edges, $V$ the number of vertices, and $N$ the number of connected components of $\Gamma(C).$ \label{graphcohomtheorem}
\end{theorem}
\begin{proof}
Using \cite[Chapter 9.2, Proposition 5]{BLR} and Lemma \ref{homeographlem} above, we may assume that $C$ is reduced. Let $\widetilde{C}$ and $C^{\mathrm{sn}}$ be the normalisation and the seminormalisation of $C,$ respectively. Because the map $C^{\mathrm{sn}} \to C$ is an homeomorphism, Lemma \ref{homeographlem} shows that $\Gamma(C)$ and $\Gamma(C^{\mathrm{sn}})$ are isomorphic. By Theorem \ref{Structurethm} and Lemma \ref{alternatingsplitlem}, the dimension of the maximal split tori inside $\Pic^0_{C/\kappa}$ and $\Pic^0_{C^{\mathrm{sn}}/\kappa}$ coincide. Hence we may assume that $C$ is seminormal. It follows from the proof of \cite[Theorem 2.24(ii)]{Ov} that the morphism $\widetilde{C}\to C=C^{\mathrm{sn}}$ can be factorised as 
$\widetilde{C}=C_1 \to C_2 \to ... \to C_n=C,$ such that, for each $i=1,..., n-1,$ there is a closed point $x_{i+1}\in C_{i+1},$ a reduced $\kappa(x_{i+1})$-algebra $A_i,$ and a closed immersion $\Spec A_i\to C_i$ such that the diagram
$$\begin{CD}
C_i@>>>C_{i+1}\\
@AAA@AAA\\
\Spec A_i @>>> \Spec \kappa(x_{i+1})
\end{CD}$$
is co-Cartesian. Moreover, we may assume that $A_i$ is either isomorphic to $\kappa(x_{i+1}) \times \kappa(x_{i+1})$ or a field extension of $\kappa(x_{i+1})$ \cite[Corollary 2.23]{Ov}.  We shall prove the Theorem by induction on $n.$ \\
If $n=0,$ then $C$ is regular and $\Gamma(C)$ contains no edges. In particular, $H_1(\Gamma(C), \Q)=0,$ so we must show that $\Pic^0_{C/\kappa}$ contains no split torus. We may replace $\kappa$ by $\kappa\alg$ and $C$ by $(C_{\kappa\alg})_{\mathrm{red}}$ (Lemma \ref{alternatingsplitlem} and \cite[Chapter 9.2, Proposition 5]{BLR}). But $C$ is geometrically unibranch\footnote{In particular, $C_{\kappa\alg}$ is unibranch. This is not immediate from the definition, but can be seen as follows: Let $x$ be a closed point of $C_{\kappa\alg}$ with image $y$ in $C_{\kappa\sep}$ and image $z$ in $C.$ By assumption, the punctured spectrum of $\Og_{C, z}^{\mathrm{sh}} = \Og_{C_{\kappa\sep}, y}^{\mathrm{h}}$ is connected \cite[ Tag 0BQ4]{Stacks}. However, since $\Og_{C_{\kappa\alg}, y}^{\mathrm{h}} = \Og_{C_{\kappa\sep},y}^{\mathrm{h}}\otimes_{\kappa\sep}\kappa\alg,$ those two rings have homeomorphic punctured spectra. }\cite[Tag 0BQ3]{Stacks}, so in the factorisation of the normalisation map $(C_{\kappa\alg})_{\mathrm{red}}^{\sim} \to (C_{\kappa\alg})_{\mathrm{red}}$ (as in \cite[Theorem 2.24]{Ov}), only push-outs along the map $\Spec \kappa\alg[\epsilon]/\langle \epsilon ^2 \rangle \to \Spec \kappa\alg$ can appear. Hence the dimensions of the maximal split tori inside $\Pic^0_{(C_{\kappa\alg})_{\mathrm{red}}/\kappa\alg}$ and $\Pic^0_{(C_{\kappa\alg})_{\mathrm{red}}^{\sim}/\kappa\alg}$ coincide. But the latter algebraic group contains no torus by \cite[Proposition 2.32]{Ov}. \\
Now let $\sigma_j$ be the dimension of the maximal split torus inside $\Pic^0_{C_j/\kappa}$ for each $j=1,..., n.$ For the induction step, we must show that, for $i=1,..., n-1,$ we have
\begin{align} \sigma_{i+1}-\sigma_i=\dim_{\Q} H_1(\Gamma(C_{i+1}), \Q) - \dim_{\Q} H_1(\Gamma(C_i), \Q). \label{quantity}\end{align} Several cases must now be considered: \\
\it 1st case: \rm The ring $A_i$ is a field extension of $\kappa(x_{i+1}).$ In this case, the map $C_i \to C_{i+1}$ is an homeomorphism, so $\Gamma(C_i)\cong \Gamma(C_{i+1})$ by Lemma \ref{homeographlem}. In particular, the quantity on the right vanishes.  Now recall the exact sequence
\begin{align}
0 &\nonumber\to \Res_{\Gamma(C_{i+1}, \Og_{C_{i+1}})/\kappa}\Gm \to \Res_{\Gamma(C_{i}, \Og_{C_{i}})/\kappa}\Gm \to \Res_{A_i/\kappa}\Gm / \Res_{\kappa(x_{i+1})/\kappa}\Gm \\
 &  \to \Pic^0_{C_{i+1}/\kappa} \to \Pic^0_{C_i/\kappa} \to 0 \label{exseqX}
\end{align}
from \cite[Proposition 2.30]{Ov}. Let $s_1,..., s_5$ be the dimensions of the maximal split tori in the algebraic groups appearing in the sequence (\ref{exseqX}) (without the trivial ones, and in the order in which they appear in the sequence). Then $s_1$ and $s_2$ are equal to the numbers of connected components of $C_{i+1}$ and $C_i,$ respectively. This follows from the fact that, for any finite field extension $\kappa\subseteq \ell,$ the canonical map $\Gm \to \Res_{\ell/\kappa}\Gm$ identifies $\Gm$ with the maximal split torus inside $\Res_{\ell/\kappa}\Gm.$ In particular, $s_1=s_2$ in the case we are considering. Using that the map
$$\Hom_{\kappa}(\Gm, \Res_{\kappa(x_{i+1})/\kappa}\Gm) \to \Hom_{\kappa}(\Gm, \Res_{A_i/\kappa}\Gm)$$ is an isomorphism and that $\Ext^1_{\kappa}(\Gm, \Res_{\kappa(x_{i+1})/\kappa}\Gm) = \Ext^1_{\kappa(x_{i+1})}(\Gm, \Gm) = 0,$ we see that $\Res_{A_i/\kappa}\Gm / \Res_{\kappa(x_{i+1})/\kappa}\Gm$ contains no split torus, so $s_3=0.$ Hence Lemma \ref{alternatingsplitlem} implies that the quantity on the left in (\ref{quantity}) vanishes as well. \\
\it 2nd case: \rm Now we consider the case where $A_i \cong \kappa(x_{i+1}) \times \kappa(x_{i+1}).$ Then there is an isomorphism 
$$\Res_{A_i/\kappa}\Gm / \Res_{\kappa(x_{i+1})/\kappa}\Gm \cong \Res_{\kappa(x_{i+1})/\kappa}\Gm,$$ so $s_3=1.$
We must once more distinguish three cases: \\
\it Case 2.A: \rm  The two connected components of $\Spec A_i$ map to distinct connected components of $C_i.$ Then we have $s_1-s_2+s_3 = 0,$ so Lemma \ref{alternatingsplitlem} shows that the quantity on the left in (\ref{quantity}) vanishes. Observe that the number of edges which are not loops in $C_i$ (resp. $C_{i+1}$) is given by the expression
$$\sum_{x} (\#\Sigma(x)-1),$$
where $x$ runs through all intersection points of $C_i$ (resp. $C_{i+1}$). It is clear that the numbers of loops in $\Gamma(C_i)$ and $\Gamma(C_{i+1})$ are the same. Depending on whether the two points of $\Spec A_i$ map to intersection points or non-intersection points in $C_i,$ we distinguish three cases, in each of which we see that the graph $\Gamma(C_{i+1})$ has one edge more than $\Gamma(C_i)$ and one fewer connected component (but the same number of vertices). Therefore the quantity on the right in (\ref{quantity}) vanishes as well. \\
\it Case 2.B: \rm The two connected components of $\Spec A_i$ map to the same connected component of $C_i,$ but there is no irreducible component containing both of them. Then $s_1-s_2+s_3=1.$ The same argument as in case 2.A shows that the graph $\Gamma(C_{i+1})$ has one edge more than $\Gamma(C_i),$ but the same number of vertices and connected components. Hence the quantities in (\ref{quantity}) coincide once again. \\
\it Case 2.C: \rm There exist irreducible components of $C_i$ containing both connected components of $\Spec A_i.$ If at most one of the images of the two points of $\Spec A_i$ is an intersection point, $\Gamma(C_{i+1})$ arises from $\Gamma(C_i)$ by adding an additional loop at precisely one vertex. Suppose therefore that both images $x$ and $y$ of the points of $\Spec A_i$ are intersection points, and denote their image in $C_{i+1}$ by $z.$ Then 
$\Sigma(z) = \Sigma(x) \cup \Sigma(y),$
so 
$$\#\Sigma(z)-1 = (\#\Sigma(x)-1) + (\#\Sigma(y) -1) - \#(\Sigma(x)\cap \Sigma(y)) +1.$$ Moreover, $\Gamma(C_{i+1})$ contains an additional loop at each vertex contained in $\Sigma(x) \cap \Sigma(y),$ which implies that $\Gamma(C_{i+1})$ has precisely one edge more than $\Gamma(C_i),$ but the same number of vertices and irreducible components. 
Hence the two quantities in (\ref{quantity}) coincide in every case and the proof is complete. 
\end{proof} \\
$\mathbf{Remark}.$ The theorem we just proved is valid without assuming that $C$ is reduced, and without any restrictions on $\delta(\kappa).$ Note, moreover, that the isomorphism type of $\Gamma(C)$ depends upon the ordering on the set of irreducible components of $C,$ whereas the homotopy type of the associated cell complex does not. This follows from the fact that $\dim_{\Q} H_1(\Gamma(C), \Q)$ only depends upon $C$ as we have just seen, and this dimension classifies a finite connected (multi)graph up to homotopy.
\begin{corollary}
Keep the notation from Theorem \ref{Raynaudtypethm}, and assume moreover that the equivalent conditions (i) - (iii) listed there are satisfied. Then the following are equivalent: \\
(a) $\Pic^0_{C/K}$ admits a Néron model (i. e. a Néron lft-model of finite type) over $R,$ and \\
(b) the graph $\Gamma(C_{\widehat{K}^{\mathrm{sh}}})$ is a forest. \\
If, moreover, $R$ is excellent, these conditions are equivalent to \\
(b') the graph $\Gamma(C_{K^{\mathrm{sh}}})$ is a forest. \label{RaynaudGenCor}
\end{corollary}
\begin{proof}
By \cite[Chapter 10.2, Theorem 1]{BLR}, (a) is equivalent to $s=0,$ where $s$ is the dimension of the maximal split torus inside $\Pic^0_{C_{\widehat{K}^{\mathrm{sh}}/\widehat{K}^{\mathrm{sh}}}}.$ This is equivalent to (b) by Theorem \ref{graphcohomtheorem} together with the fact that a finite (multi)graph $\Gamma$ is a forest if and only if $H_1(\Gamma, \Q)=0.$ The equivalence (a) $\iff$ (b') if $R$ is excellent follows in an entirely analogous manner.
\end{proof}

\textsc{Mathematisches Institut der Heinrich-Heine-Universität Düsseldorf, Universitätsstr. 1, 40225 Düsseldorf, Germany} \\
\it E-mail address: \rm \texttt{otto.overkamp@uni-duesseldorf.de}\


\begin{thebibliography}{10}



\bibitem{An}
Anantharaman, S. 
\textit{Schémas en groupes, espaces homogènes et espaces algébriques sur une base de dimension 1}. Bull. Soc. Math. France, Mémoire 33 (1973), pp. 5-79.

\bibitem{BLR}
Bosch, S., L\"utkebohmert, W., Raynaud, M.
\textit{Néron models}. Ergeb. Math. Grenzgeb., Springer-Verlag, Berlin, Heidelberg, 1990.

\bibitem{Chai}
Chai, C.-L.
\textit{Néron models for semiabelian varieties: Congruence and change of base field}. Asian J. Math. 4:4, pp. 715-736, 2000.

\bibitem{CY}
Chai, C.-L., Yu, J.-K.
\textit{Congruences of Néron models for tori and the Artin conductor}. Ann. of Math. 154, pp. 347-382, 2001.

\bibitem{CLN}
Cluckers, R., Loeser, F., Nicaise, J.
\textit{Chai's conjecture and Fubini properties of dimensional motivic integration}. Algebra Number Theory, Vol. 7, No. 4, pp. 893-915, 2013. 

\bibitem{CS}
Colliot-Thélène, J.-L., Skorobogatov, A.
\textit{The Brauer-Grothendieck Group}. Ergeb. Math. Grenzgeb., 3. Folge, A series of Modern Surveys in Mathematics 71, Springer-Verlag, 2021.

\bibitem{CGP}
Conrad, B., Gabber, O., Prasad, G.
\textit{Pseudo-reductive Groups}. 2nd ed. New Mathematical monographs 26, Cambridge University Press, 2015.

\bibitem{SGA3}
Demazure, M., Grothendieck, A. Eds. 
\textit{Schémas en groupes I, II, III (SGA 3)}, Springer Lecture Notes in Math. 151, 152, 153 (1970); revised version edited by P. Gille and P. Polo, vols. I and III, Soc. Math. de France, 2011.

\bibitem{HN}
Halle, L. H., Nicaise, J.
\textit{Néron Models and Base Change}. Springer Lecture Notes in Math. 2156, Springer-Verlag, 2016.

\bibitem{Liu}
Liu, Q.
\textit{Algebraic Geometry and Arithmetic Curves}. Translated by R. Erné. Oxford Graduate Texts in Mathematics, Oxford University Press, 2002.

\bibitem{LLR}
Liu, Q., Lorenzini, D., Raynaud, M., 
\textit{Néron models, Lie algebras, and reduction of curves of genus one}. Invent. Math. 157, pp. 455-518, 2004.

\bibitem{LLRII}
Liu, Q., Lorenzini, D., Raynaud, M., 
\textit{Corrigendum to Néron models, Lie algebras, and reduction of curves of genus one and The Brauer group of a surface}. Invent. Math. 214, pp. 593-604, 2018.

\bibitem{OvI}
Overkamp, O. 
\textit{Jumps and Motivic Invariants of Semiabelian Jacobians}. Int. Math. Res. Not., Issue 20, pp. 6437-6479, 2019.

\bibitem{Ov}
Overkamp, O.
\textit{On Jacobians of geometrically reduced curves and their Néron models}. Trans. Amer. Math. Soc., Vol. 377, Nr. 8, pp. 5863-5903, 2024.

\bibitem{Pép}
Pépin, C.
\textit{Modèles semi-factoriels et modèles de Néron}. Math. Ann. 335, pp. 147-185, 2013.

\bibitem{Raynaud}
Raynaud, M.
\textit{Anneaux excellents}. Rédigé par Y. Laszlo. In \it Travaux de Gabber sur l'uniformisation locale et la cohomologie étale des schémas quasi-excellents. (Séminaire à l'École polytechnique 2006-2008). \rm Astérisque, Tome 363-364, Société Mathématique de France, 2014. 

\bibitem{RaynaudPic}
Raynaud, M.
\textit{Spécialisation du foncteur de Picard}. Publ. Math. IHES, tome 38, pp. 27-76, 1970.

\bibitem{Schr}
Schröer, S.
\textit{Fibrations whose geometric fibres are nonreduced}. Nagoya Math J. 200, pp. 35-57, 2010.

\bibitem{Stacks}
Authors of the Stacks project.
\textit{Stacks project}. Columbia University.
\end{thebibliography}
\end{document}